\documentclass[12pt,a4paper]{article} 
\usepackage{amsmath}
\usepackage{amsthm}
\usepackage{amsfonts}
\usepackage{amssymb}
\usepackage{bussproofs}
\author{Amir Akbar Tabatabai}
\theoremstyle{plain} 
\newtheorem{thm}{Theorem}[section]

\newtheorem{lem}[thm]{Lemma}

\theoremstyle{definition}
\newtheorem{dfn}[thm]{Definition}
\newtheorem{exam}[thm]{Example}
\newtheorem{rem}[thm]{Remark}

\def\PA{\mathrm{PA}}
\def\Pr{\mathrm{Pr}}
\def\Prf{\mathrm{Prf}}

\def\S4{\mathrm{S4}}
\def\Cons{\mathrm{Cons}}
\def\Rfn{\mathrm{Rfn}}

\begin{document}
\title{Provability Logics of Hierarchies} 

\author{Amirhossein Akbar Tabatabai \footnote{The author is supported by the ERC Advanced Grant 339691 (FEALORA)}\\
Institute of Mathematics\\
Academy of Sciences of the Czech Republic\\
tabatabai@math.cas.cz}

\date{February 9, 2017 }

\maketitle

\begin{abstract}
The branch of provability logic investigates the provability-based behavior of the mathematical theories. In a more precise way, it studies the relation between a mathematical theory $T$ and a modal logic $L$ via the provability interpretation which interprets the modality as the provability predicate of $T$. In this paper we will extend this relation to investigate the provability-based behavior of a \textit{hierarchy} of theories. More precisely, using the modal language with infinitely many modalities, $\{\Box_n\}_{n=0}^{\infty}$, we will define the hierarchical counterparts of some of the classical modal theories such as $\mathbf{K4}$, $\mathbf{KD4}$, $\mathbf{GL}$ and $\mathbf{S4}$. Then we will define their canonical provability interpretations and their corresponding soundness-completeness theorems.
\end{abstract}

\section{Introduction}
Provability logic is a branch of mathematical logic which investigates the provability-based behavior of the mathematical theories. In a more precise way, it studies the relation between a theory $T$ and a modal logic $L$ via the provability interpretation which interprets $\Box$ in the language of $L$ as the provability predicate for the theory $T$. The key example of this relation is the relation between the theory $\mathbf{PA}$ and the modal logic $\mathbf{GL}$ presented by R. Solovay in \cite{So}. Inspired by this seminal work, these kinds of relations have been fully investigated in terms of different aspects. But in spite of the extensive work, it seems that there are still some problems unsolved. The main theme is the following: There are some modal theories such as $\mathbf{K4}$, $\mathbf{KD4}$ or $\mathbf{S4}$ which admit some kinds of informal provability interpretations. However, none of these logics is a provability logic for any theory in the usual formal sense. The problem is, how it is possible to formalize those intuitive provability interpretations to widen the horizon and see these modal theories as provability logics. \\
Let us illuminate this problem by a classical example. The best example is G\"{o}del's problem of finding a provability interpretation for $\mathbf{S4}$, proposed in his paper \cite{G}. Think of the axioms of the system $\mathbf{S4}$. It seems that all of them are valid under the intuitive interpretation of $\Box$ as the informal provability predicate. The axiom $\mathbf{K}: \Box(A \rightarrow B) \rightarrow (\Box A \rightarrow \Box B)$ means that the provability predicate is closed under modus ponens. The axiom $\mathbf{4}: \Box A \rightarrow \Box \Box A$ states that ``the provability of a provable statement is also provable" which seems a reasonable condition to have and finally $\mathbf{T}: \Box A \rightarrow A$ states that the proofs are sound. Therefore, it seems that $\mathbf{S4}$ is a valid theory for the concept of provability. However, $\mathbf{S4}$ is not a provability logic in the usual sense, because if it is a provability logic of a theory $T$, then $\neg \Box \bot \wedge \Box \neg \Box \bot$ should be true under the provability interpretation. This means that the statement $\Cons(T) \wedge \Pr_T(\Cons(T))$ holds, which contradicts with G\"{o}del's second incompleteness theorem. Therefore, the following question emerges: If the usual provability interpretation does not work for $\mathbf{S4}$, then what is the formalization of the intuitive interpretation we used before?\\

To solve these kinds of problems, in \cite{AK} we introduced a way of extending the framework of provability logic to capture more theories, including $\mathbf{S4}$. The main idea is using a hierarchy of theories instead of just one theory. The explanation is the following: In the modal language there are nested modalieties which intuitively capture the nested use of the provability predicate in mathematics; statements like provability of $p$, provability of ``provability of $p$" and so on. These different layers of provability naturally refer to different layers of theories, meta-theories, meta-meta-theories and so on. But the usual provability interpretation reads all of them as the provability predicate of a fixed theory. Philosophically speaking, we know that there is no reason to assume that all layers of our meta-theories are the same. Quite the contrary, in actual practice of mathematical logic, sometimes we need to have more powerful meta-theories to investigate the behavior of the theory. Therefore, we proposed using a hierarchy of theories to formalize the different layers of meta-theories instead of using just one theory for all the levels of the concept of provability. Following that approach and using some natural classes of the hierarchies of arithmetical theories, we found some natural interpretations for some modal logics such as $\mathbf{K4}$, $\mathbf{KD4}$ and $\mathbf{S4}$.\\

This framework extension is clearly useful for the problem of finding a provability interpretation for a given modal logic, but it also proposes a brand new problem, the converse of the first problem which is the problem of finding the provability logic of a given class of hierarchies. Trivially, we can interpret our work \cite{AK} as a way to answer this question, but there are some technical problems which make the usual language of modal logics quite inappropriate for this purpose. The reason is that in the language of modal logic we have just one modality and we know that there is no canonical way to interpret the different occurrences of this modality as the different provability predicates in the hierarchy. Hence, it seems that the usual language of modal logics is not a natural choice if we want to capture the provability-based behavior of a hierarchy. To handle this problem, it seems that we need a modal language with infinitely many modalities to capture the different layers of the meta-theories' hierarchy. \\

In this paper we will follow this poly-modal approach. In fact, using the language mentioned above we will introduce the hierarchical counterparts of some of the usual modal logics, such as $\mathbf{K4}$, $\mathbf{KD4}$ and $\mathbf{S4}$. Then we will introduce a natural provability interpretation for these new logics and finally, we will prove the soundness-completeness theorems for this interpretation.
\section{Preliminaries}
Our main strategy to prove the soundness-completeness result is reducing the completeness of the new theories to the completeness of the usual modal theories proved in \cite{AK}. To follow this strategy, we need some of the notions and theorems of \cite{AK}. In this section we will explain them.\\

The first key ingredient is the notion of a provability model as a natural model to capture our intuitive notion of the world to evaluate statements and the notion of the hierarchy of the different layers of meta-theories. (For more detailed explanation, see \cite{AK}.)
\begin{dfn}\label{t1-1} 
A provability model is a pair $(M, \{T_n\}_{n=0}^{\infty})$ where $M$ is a model of $I\Sigma_1$ and $\{T_n\}_{n=0}^{\infty}$ is a hierarchy of arithmetical r.e. theories such that for any $n$, $I\Sigma_1 \subseteq T_n \subseteq T_{n+1}$ provably in $I\Sigma_1$.
\end{dfn}
The next ingredient is the notion of witness. Informally speaking, it is just a way of assigning numbers to boxes in a formula. The goal is assigning theories in the hierarchy to boxes. The condition is that the number for the outer box should be greater than the number for inner boxes. This condition captures the fact that the outer box refers to the meta-theory of the theories used for the inner boxes.
\begin{dfn}\label{t1-2} 
Let $w$ be a sequence of natural numbers and $A$ be a modal formula. Then the relation $w \Vdash A$, which means $w$ is a witness for $A$, is inductively defined as follows:
\begin{itemize}
\item[$\bullet$]
If $A$ is an atom, $() \Vdash A$.
\item[$\bullet$]
If $A=B \circ C$, then $(w_1, w_2) \Vdash A$ if $w_1 \Vdash B$ and $w_2 \Vdash C$ for $\circ \in \{\wedge, \vee, \rightarrow\}$
\item[$\bullet$]
If $A=\neg B$, then $w \Vdash A$ if $w \Vdash B$.
\item[$\bullet$]
If $A=\square B $, then $(n, w) \Vdash A$ if $w \Vdash B$ and $n > m$ for all $m$ which are appeared in $w$.
\end{itemize}
Moreover, if $\Gamma$ is a sequence of modal formulas, by a witness for $\Gamma$ we mean, a sequence of witnesses, such that any witness $w_i$ in the sequence, is a witness for $A_i$ in $\Gamma$.
\end{dfn}
The next concept that we need is the notion of evaluation.
\begin{dfn}\label{t1-3} 
Let $w$ be a witness for $A$ and $\sigma$ be an arithmetical substitution which assigns an arithmetical sentence to any propositional variable. And also let $(M, \{T_n\}_{n=0}^{\infty})$ be a provability model. By $A^{\sigma}(w)$ we mean an arithmetical sentence which is resulted by substituting the variables by $\sigma$ and interpreting any box as the provability predicate of $T_n$ if the corresponding number in the witness for this box was $n$. The interpretation of the boolean connectives are themselves. Moreover, if $\Gamma$ is a sequence of modal formulas $A_i$, and $w=(w_i)_{i}$ is its witness, by $\Gamma^{\sigma}(w)$ we mean the sequence of $A_i^{\sigma}(w_i)$.
\end{dfn}
And the notion of satisfaction:
\begin{dfn}\label{t1-4} 
A sequent $\Gamma \Rightarrow \Delta$ is true in $(M, \{T_n\}_{n=0}^{\infty})$ when there are witnesses $u$ and $v$ for $\Gamma$ and $\Delta$ respectively, such that for any arithmetical substitution $\sigma$, $M \vDash \Gamma^{\sigma}(u) \Rightarrow \Delta^{\sigma}(v)$. Moreover, we say a sequent $\Gamma \Rightarrow \Delta$ is true in a class of models $\mathcal{C}$, when there are uniform witnesses for all models. In a more precise way, we write $\mathcal{C} \vDash \Gamma \Rightarrow \Delta$, if there are witnesses $u$ and $v$ such that for all arithmetical substitutions $\sigma$ and all provability models $(M, \{T_n\}_{n=0}^{\infty})$ in $\mathcal{C}$, $M \vDash \Gamma^{\sigma}(u) \Rightarrow \Delta^{\sigma}(v)$.
\end{dfn}
\begin{rem}\label{t1-5}
The Definition \ref{t1-4} is actually a weaker version of what we used in \cite{AK}. The full definition is more complicated, but since we just need the completeness part of these interpretations, this weaker version would be enough for our purpose.
\end{rem}
To have a completeness result we need some classes of provability models. In the following, we will define some of the natural ones:
\begin{dfn}\label{t1-6}
\begin{itemize}
\item[$(i)$]
The class of all provability models will be denoted by $\mathbf{PrM}$.
\item[$(ii)$]
A provability model $(M,\{T_n\}_{n=0}^{\infty})$ is called consistent if for any $n$, $M$ thinks that $T_n$ is consistent and $T_{n+1} \vdash \Cons(T_n)$, i.e. $M \vDash \Cons(T_n)$ and $M \vDash \Pr_{T_{n+1}}(\Cons(T_n))$. The class of all consistent provability models will be denoted by $\mathbf{Cons}$.
\item[$(iii)$]
A provability model $(M,\{T_n\}_{n=0}^{\infty})$ is reflexive if for any $n$, $M$ thinks that $T_n$ is sound and $T_{n+1} \vdash \Rfn(T_n)$, i.e. $M \vDash \Pr_{T_n}(A) \rightarrow A$ and $M \vDash \Pr_{T_{n+1}}(\Pr_{T_n}(A) \rightarrow A)$ for any sentence $A$. The class of all reflexive provability models will be denoted by $\mathbf{Ref}$.
\item[$(iii')$]
A hierarchy $\{T_n\}_{n=0}^{\infty}$ of theories is called uniform if there exists a $\Sigma_1$ formula $\Prf(x, y, z)$ such that for any $n$, $m$ and $A$, $\Prf(n, m, \lceil A \rceil)$ iff $m$ is a code of a proof for $A$ in $T_n$. The hierarchy is called uniformly increasing if it is a uniform hierarchy and also we have $I\Sigma_1 \subseteq T_0$ provably in $I\Sigma_1$ and $I\Sigma_1 \vdash \forall x \forall z (\exists y \; \Prf(x, y, z) \rightarrow \exists w \; \Prf(x+1, w, z))$. And finally a hierarchy is called uniformly reflexive if it is a uniformly increasing hierarchy such that for any formula $A$, $I\Sigma_1 \vdash \forall x  \exists y \; \Prf(x+1, y, \exists w \; \Prf(x, w, A) \rightarrow A)$. If $\{T_n\}_{n=0}^{\infty}$ is uniformly reflexive and $M \vDash \bigcup_{n=0}^{\infty}T_n$, the provability model $(M,\{T_n\}_{n=0}^{\infty})$ is called uniformly reflexive. The class of all uniformly reflexive models will be denoted by $\mathbf{uRef}$.
\item[$(iv)$]
A provability model, $(M, \{T_n\}_{n=0}^{\infty})$ is constant if for any $n$ and $m$, $(M, \{T_n\}_{n=0}^{\infty})$ thinks that $T_n=T_m$, i.e. $M \vDash \Pr_{T_m}(A) \leftrightarrow \Pr_{T_n}(A)$ and $M \vDash \Pr_{T_0}(\Pr_{T_m}(A) \leftrightarrow \Pr_{T_n}(A))$ for any sentence $A$. The class of all constant provability models will be denoted by $\mathbf{Cst}$.
\end{itemize}
\end{dfn}
And finally we have the following completeness theorems for the usual modal logics:
\begin{thm}\label{t1-7}
\begin{itemize}
\item[$(i)$]
If $\mathbf{PrM} \vDash \Gamma \Rightarrow A$, then $\Gamma \vdash_{\mathbf{K4}} A$.
\item[$(ii)$]
If $\mathbf{Cons} \vDash \Gamma \Rightarrow A$, then $\Gamma \vdash_{\mathbf{KD4}} A$.
\item[$(iii)$]
Let $ \{T_n\}_{n=0}^{\infty}$ be a uniformly reflexive hierarchy of sound theories. Then there exists an arithmetical substitution $*$, such that for any modal sequent $\Gamma \Rightarrow A$, if there exist witnesses $u$ and $v$ such that for all $M \vDash \bigcup_nT_n$, $(M, \{T_n\}_{n=0}^{\infty}) \vDash \Gamma^{*}(u) \Rightarrow A^{*}(v)$, then $ \Gamma \vdash_{\mathbf{S4}} A$. Moreover, If $\mathbf{Ref} \vDash \Gamma \Rightarrow A$, then $\Gamma \vdash_{\mathbf{S4}} A$.
\item[$(vi)$]
Let $I \Sigma_1 \subseteq T$ be an r.e. $\Sigma_1$-sound theory and $\{T_n\}_{n=0}^{\infty}$ be a hierarchy of theories such that for any $n$, $T_n=T$, then there is an arithmetical substitution $*$ such that for any modal sequent $\Gamma \Rightarrow A$, if for all $M \vDash T$, we have $(M, \{T_n\}_{n=0}^{\infty}) \vDash \Gamma \Rightarrow A$, then $\Gamma \vdash_\mathbf{GL} A$. And especially, if $\mathbf{Cst} \vDash \Gamma \Rightarrow A$ then $\Gamma \vdash_\mathbf{GL} A$.
\end{itemize}
\end{thm}
As the last word in this section, let us remind you the three sequent calculi for the modal logics $\mathbf{K4}$, $\mathbf{KD4}$ and $\mathbf{S4}$. We will need them in the next section. Consider the following rules:
\begin{flushleft}
	\textbf{Axioms:}
\end{flushleft}
\begin{center}
	\begin{tabular}{c c}
		\AxiomC{$  A \Rightarrow A$ }
		\DisplayProof 
			&
		\AxiomC{$ \bot \Rightarrow  $}
		\DisplayProof
	\end{tabular}

\end{center}
\begin{flushleft}
 		\textbf{Structural Rules:}
\end{flushleft}
\begin{center}
	\begin{tabular}{c}
		\begin{tabular}{c c}
		\AxiomC{$\Gamma  \Rightarrow \Delta$}
		\LeftLabel{\tiny{$ (wL) $}}
		\UnaryInfC{$\Gamma,  A  \Rightarrow \Delta$}
		\DisplayProof
			&
		\AxiomC{$\Gamma  \Rightarrow \Delta$}
		\LeftLabel{\tiny{$ ( wR) $}}
		\UnaryInfC{$\Gamma \Rightarrow  \Delta, A$}
		\DisplayProof
		\end{tabular}
			\\[3 ex]
			\begin{tabular}{c c}
		\AxiomC{$\Gamma, A, A \Rightarrow \Delta$}
		\LeftLabel{\tiny{$ (cL) $}}
		\UnaryInfC{$\Gamma,  A  \Rightarrow \Delta$}
		\DisplayProof
		    &
		\AxiomC{$\Gamma \Rightarrow \Delta, A, A$}
		\LeftLabel{\tiny{$ (cR) $}}
		\UnaryInfC{$\Gamma \Rightarrow \Delta, A$}
		\DisplayProof
		\end{tabular}
        \\[3 ex]
	    
	    \AxiomC{$\Gamma_0 \Rightarrow \Delta_0, A$}
	    \AxiomC{$\Gamma_1, A \Rightarrow \Delta_1$}
		\LeftLabel{\tiny{$ (cut) $}}
		\BinaryInfC{$\Gamma_0, \Gamma_1 \Rightarrow \Delta_0, \Delta_1$}
		\DisplayProof
	\end{tabular}
\end{center}		
\begin{flushleft}
  		\textbf{Propositional Rules:}
\end{flushleft}
\begin{center}
  	\begin{tabular}{c c}
  		\AxiomC{$\Gamma_0, A  \Rightarrow \Delta_0 $}
  		\AxiomC{$\Gamma_1, B  \Rightarrow \Delta_1$}
  		\LeftLabel{{\tiny $\vee L$}} 
  		\BinaryInfC{$ \Gamma_0, \Gamma_1, A \lor B \Rightarrow \Delta_0, \Delta_1 $}
  		\DisplayProof
	  		&
	   	\AxiomC{$ \Gamma \Rightarrow \Delta, A_i$}
   		\RightLabel{{\tiny $ (i=0, 1) $}}
   		\LeftLabel{{\tiny $\vee R$}} 
   		\UnaryInfC{$ \Gamma \Rightarrow \Delta, A_0 \lor A_1$}
   		\DisplayProof
	   		\\[3 ex]
   		\AxiomC{$ \Gamma, A_i \Rightarrow \Delta$}
   		\RightLabel{{\tiny $ (i=0, 1) $}} 
   		\LeftLabel{{\tiny $\wedge L$}}  		
   		\UnaryInfC{$ \Gamma, A_0 \land A_1 \Rightarrow \Delta, C $}
   		\DisplayProof
	   		&
   		\AxiomC{$\Gamma_0  \Rightarrow \Delta_0, A$}
   		\AxiomC{$\Gamma_1  \Rightarrow  \Delta_1, B$}
   		\LeftLabel{{\tiny $\wedge R$}} 
   		\BinaryInfC{$ \Gamma_0, \Gamma_1 \Rightarrow \Delta_0, \Delta_1, A \land B $}
   		\DisplayProof
   			\\[3 ex]
   		\AxiomC{$ \Gamma_0 \Rightarrow A, \Delta_0 $}
  		\AxiomC{$ \Gamma_1, B \Rightarrow \Delta_1, C $}
  		\LeftLabel{{\tiny $\rightarrow L$}} 
   		\BinaryInfC{$ \Gamma_0, \Gamma_1, A \rightarrow B \Rightarrow \Delta_0, \Delta_1, C$}
   		\DisplayProof
   			&
   		\AxiomC{$ \Gamma, A \Rightarrow B, \Delta $}
   		\LeftLabel{{\tiny $\rightarrow R$}} 
   		\UnaryInfC{$ \Gamma \Rightarrow \Delta, A \rightarrow B$}
   		\DisplayProof
   		\\[3 ex]
   		\AxiomC{$ \Gamma \Rightarrow \Delta, A $}
   		\LeftLabel{\tiny {$\neg L$}} 
   		\UnaryInfC{$ \Gamma, \neg A \Rightarrow \Delta$}
   		\DisplayProof
   			&
   		\AxiomC{$ \Gamma, A \Rightarrow \Delta $}
   		\LeftLabel{{\tiny $\neg R$}} 
   		\UnaryInfC{$ \Gamma \Rightarrow \Delta, \neg A$}
   		\DisplayProof
	\end{tabular}
\end{center}

\begin{flushleft}
 		\textbf{Modal Rules:}
\end{flushleft}
\begin{center}
  	\begin{tabular}{c c c}
		\AxiomC{$ \Gamma, \Box \Gamma \Rightarrow A$}
		\LeftLabel{\tiny{$\Box_4 R$}}
		\UnaryInfC{$\Box \Gamma \Rightarrow \Box A$}
		\DisplayProof
		&
		\AxiomC{$ \Gamma, \Box \Gamma \Rightarrow $}
		\LeftLabel{\tiny{$\Box_D R$}}
		\UnaryInfC{$\Box \Gamma \Rightarrow $}
		\DisplayProof
		&
		\AxiomC{$ \Box \Gamma \Rightarrow A$}
		\LeftLabel{\tiny{$\Box_S R$}}
		\UnaryInfC{$\Box \Gamma \Rightarrow \Box A$}
		\DisplayProof
		\end{tabular}
\end{center}
\begin{center}
  	\begin{tabular}{c c}	
        \AxiomC{$ \Gamma, A \Rightarrow \Delta$}
		\LeftLabel{\tiny{$\Box L$}}
		\UnaryInfC{$\Gamma, \Box A \Rightarrow \Delta$}
		\DisplayProof
		\\[2 ex]
	\end{tabular}
\end{center}
The system $G(\mathbf{K4})$ is the system that consists of the axioms, structural rules and propositional rules and the modal rule $\Box_{4} R$. $G(\mathbf{KD4})$ is $G(\mathbf{K4})$ plus the rule $\Box_{D} R$ and finally, $G(\mathbf{S4})$ is the system $G(\mathbf{K4})$ when we replace the rule $\Box_{4} R$ by $\Box_{S} R$ and add the rule $\Box L$. All of these systems have the cut elimination property. (See \cite{Po}).
\section{Hierarchical Modal Theories and Their Proof Theory}
In this section we will define an appropriate language to reflect a hierarchy of theories in the language of modal logics. Then we will introduce some natural modal theories to formalize different provability-based behaviors of hierarchies and finally we will investigate some of their proof-theoretic properties.
\begin{dfn}\label{t2-1}
Consider the language of modal logics with infinitely many modalities, $\{\Box_n\}_{n=0}^{\infty}$. The set of formulas in this language, $\mathcal{L}_{\infty}$, is defined as the least set of expressions which includes all atomic formulas and is closed under all boolean operations and also the following operation: If $A \in \mathcal{L}_{\infty}$ and $n$ is bigger than all indices of boxes occurred in $A$ then $\Box_n A \in \mathcal{L}_{\infty}$. In other words, $A \in \mathcal{L}_{\infty}$ if $A$ is a usual formula in the modal language and also the index of any box is bigger than the indices of all other boxes in its scope. Moreover, if $A \in \mathcal{L}_{\infty}$, by the rank of $A$, $r(A)$, we mean the biggest index occurring in $A$ and if $\Gamma \subseteq \mathcal{L}_{\infty}$, by $r(\Gamma)$ we mean the maximum of the ranks of $A \in \Gamma$.
\end{dfn}
For example $\Box _1 (\neg \Box_0 p \wedge q)$ is a formula with rank one, while the expression $\Box_1 \Box_1 p$ is not even a formula. Notice that the reason behind the limiting condition of the definition of $\mathcal{L}_{\infty}$ is that we believe in the separation of the levels of the theories and meta-theories. Therefore, referring to meta-theories in the lower theories or even the meta-theories themselves should be considered as a syntactical error. \\

For the provability-based semantics, consider the following informal definition: The provability interpretation of a formula $A$ is an arithmetical formula which interprets $\Box_n$ as the provability predicate of the theory $T_n$. Formally speaking:
\begin{dfn}\label{t2-2}
Let $(M, \{T_n\}_{n=0}^{\infty})$ be a provability model and $A \in \mathcal{L}_{\infty}$ a formula. Then by an arithmetical substitution $\sigma$ we mean a function from atomic formulas to the set of arithmetical sentences. Moreover, by $A^{\sigma}$ we mean an arithmetical sentence which is resulted by substituting the atomic formulas by $\sigma$ and interpreting any $\Box_n$ as the provability predicate of $T_n$. The interpretation of boolean connectives are themselves. In addition, if $\Gamma$ is a sequence of formulas $A_i$, by $\Gamma^{\sigma}$ we mean the sequence of $A_i^{\sigma}$.
\end{dfn}
\begin{dfn}\label{t2-3}
Let $(M, \{T_n\}_{n=0}^{\infty})$ be a provability model and $A \in \mathcal{L}_{\infty}$ a formula. Then we say $(M, \{T_n\}_{n=0}^{\infty}) \vDash A$ if for any arithmetical substitution $\sigma$, $M \vDash A^{\sigma}$. Moreover, if $\Gamma$ and $\Delta$ are sequences of formulas (not necessarily finite) and $\mathcal{C}$ a class of provability models, by $\mathcal{C} \vDash \Gamma \Rightarrow \Delta$, we mean that for any $(M, \{T_n\}_{n=0}^{\infty}) \in \mathcal{C}$, and for any arithmetical substitution $\sigma$, if $M \vDash \bigwedge \Gamma^{\sigma}$, then $M \vDash \bigvee \Delta^{\sigma}$.
\end{dfn}
Let us illuminate the definition above by an example.
\begin{exam}\label{t2-4} 
Let $(\mathbb{N}, \{T_n\}_{n=0}^{\infty})$ be a pair where $T_0=\PA$ and for any $n$, $T_{n+1}= T_n + \Rfn(T_n)$. Based on the definition, this pair is obviously a provability model. Now, we want to show that the sentence $\Box_{n+1}(\Box_n A \rightarrow A)$ is true in the model. The proof is simple: For any arithmetical substitution $\sigma$, we have $\mathbb{N} \vDash \Pr_{T_{n+1}}(\Pr_{T_n} (A^{\sigma}) \rightarrow A^{\sigma})$ since the theory $T_{n+1}$ can prove the reflection for $T_n$.
\end{exam}
It is time to define some modal theories in this language:
\begin{dfn}\label{t2-5}
Consider the following set of axiom schemes:
\begin{itemize}
\item[$(\mathbf{H})$]
$\Box_{n} A \rightarrow \Box_{n+1} A$
\item[$(\mathbf{K}_h)$]
$\Box_n (A \rightarrow B) \rightarrow (\Box_n A \rightarrow \Box_n B)$
\item[$(\mathbf{4}_h)$]
$\Box_n A \rightarrow \Box_{n+1} \Box_n A$
\item[$(\mathbf{D}_h)$]
$\neg \Box_n \bot $
\item[$(\mathbf{L}_h)$]
$\Box_{n+1}(\Box_n A \rightarrow A) \rightarrow \Box_n A$
\item[$(\mathbf{T}_h)$]
$\Box_n A \rightarrow A$
\item[$(\mathbf{5}_h)$]
$\neg \Box_n A \rightarrow \Box_{n+1} \neg \Box_n A$
\end{itemize}
Let $X$ be a set of these schemes. By $L(X)$ we mean the least set of formulas in $\mathcal{L}_{\infty}$ which contains all classical tautologies on formulas in $\mathcal{L}_{\infty}$, includes all instances of the axioms in the set $X$ for any natural number $n \geq 0$, and is closed under the following rules:
\begin{itemize}
\item[$(\mathbf{MP})$]
If $A \in L(X)$ and $A \rightarrow B \in L(X)$ then $B \in L(X)$.
\item[$(\mathbf{NC}_h)$]
If $A \in L(X)$ then $\Box_n A \in L(X)$ for any natural number $n$ greater than all the numbers occurred in $A$.
\end{itemize}
Moreover, if $\Gamma \cup \{A\} \subseteq \mathcal{L}_{\infty}$, by $\Gamma \vdash_{L(X)} A$ we mean that there exists a finite set $\Delta \subseteq \Gamma$ such that $L(X) \vdash \bigwedge \Delta \rightarrow A$.\\
Finally, we define $\mathbf{K4}_h=L(\mathbf{H}, \mathbf{K}_h, \mathbf{4}_h)$, $\mathbf{KD4}_h=L(\mathbf{H}, \mathbf{K}_h, \mathbf{4}_h, \mathbf{D}_h)$, $\mathbf{S4}_h=L(\mathbf{H}, \mathbf{K}_h, \mathbf{4}_h, \mathbf{T}_h)$, $\mathbf{GL}_h=L(\mathbf{H}, \mathbf{K}_h, \mathbf{4}_h, \mathbf{L}_h)$ and $\mathbf{KD45}_h=L(\mathbf{H}, \mathbf{K}_h, \mathbf{D}_h, \mathbf{4}_h, \mathbf{5}_h)$ and $\mathbf{S5}_h=L(\mathbf{H}, \mathbf{K}_h, \mathbf{4}_h, \mathbf{T}_h, \mathbf{5}_h)$. 
\end{dfn}
\begin{rem}\label{t2-6}
Note that from now on, all formulas are supposed to belong to the set $\mathcal{L}_{\infty}$. For instance, we can not use the axiom $\mathbf{K}_h$ for all $n$'s. $n$ should be bigger than all the indices occurred in $A$ and $B$. 
\end{rem}
In the remaining part of this section, we will investigate some of the proof-theoretical properties of the above-mentioned theories. These investigations are not complete in any sense. The reason is that we focus on the properties that somehow we need for the proof of the soundness-completeness theorems in the following sections.\\ 

First of all, we have to mention that for the soundness theorems for some of the theories, we need a cut-free sequent-style representation of the proofs. Therefore, the next step will be introducing Gentzen-style systems for our logics. To achieve this goal, consider the following set of modal rules:
\begin{center}
  	\begin{tabular}{c c c}
		\AxiomC{$ \{\sigma_r\}_{r \in R}, \{ \gamma_i, \Box_{n_i} \gamma_i\}_{i \in I} \Rightarrow A$}
		\LeftLabel{\tiny{$\Box_{4_h} R$}}
		\UnaryInfC{$\{\Box_n \sigma_r\}_{r \in R}, \{ \Box_{n_i} \gamma_i\}_{i \in I} \Rightarrow \Box_n A$}
		\DisplayProof
		&
		\AxiomC{$ \{\sigma_r\}_{r \in R}, \{ \gamma_i, \Box_{n_i} \gamma_i\}_{i \in I} \Rightarrow $}
		\LeftLabel{\tiny{$\Box_{D_h} R$}}
		\UnaryInfC{$\{\Box_n \sigma_r\}_{r \in R}, \{\Box_{n_i} \gamma_i\}_{i \in I} \Rightarrow $}
		\DisplayProof
		\\[2 ex]
		\end{tabular}
\end{center}
\begin{center}
  	\begin{tabular}{c c}	
        \AxiomC{$ \Gamma, A \Rightarrow \Delta$}
		\LeftLabel{\tiny{$\Box_h L$}}
		\UnaryInfC{$\Gamma, \Box_n A \Rightarrow \Delta$}
		\DisplayProof
		&
		\AxiomC{$ \{\sigma_r\}_{r \in R}, \{ \Box_{n_i} \gamma_i\}_{i \in I} \Rightarrow A$}
		\LeftLabel{\tiny{$\Box_{S_h} R$}}
		\UnaryInfC{$\{\Box_n \sigma_r\}_{r \in R}, \{ \Box_{n_i} \gamma_i\}_{i \in I} \Rightarrow \Box_n A$}
		\DisplayProof
		\\[2 ex]
	\end{tabular}
\end{center}
The condition of applying the rules $\Box_{4_h} R$, $\Box_{D_h} R$ and $\Box_{S_h} R$ is that for all $i \in I$, $n_i < n$.\\
The system $G(\mathbf{K4}_h)$ is the system that consists of the axioms, structural rules and propositional rules as introduced in the Preliminaries and the modal rule $\Box_{4_h} R$. $G(\mathbf{KD4}_h)$ is $G(\mathbf{K4}_h)$ plus the rule $\Box_{D_h} R$ and finally, $G(\mathbf{S4}_h)$ is the system $G(\mathbf{K4}_h)$ when we replace the rule $\Box_{4_h} R$ by $\Box_{S_h} R$ and add the rule $\Box_{h} L$.
\begin{thm}\label{t2-7}
The Systems $G(\mathbf{K4}_h)$, $G(\mathbf{KD4}_h)$ and $G(\mathbf{S4}_h)$ are equivalent to $\mathbf{K4}_h$, $\mathbf{KD4}_h$ and $\mathbf{S4}_h$, respectively.
\end{thm}
\begin{proof}
Firstly, we will show that all of these sequnet calculi are strong enough to simulate their Hilbert style counterparts. To do that, we will use induction on the length of the Hilbert style proof. First, we have to show that the axioms can be simulated. For the case of classical tautologies, we can prove all classical tautologies in the sequent systems because we have all classical propositional rules available. For the modal axioms it is enough to consider the following proof trees:
\begin{flushleft}
The proof of the axiom $\mathbf{K}_h$ in all systems:
\end{flushleft}
\begin{center}
  	\begin{tabular}{c c c}
  	
		\AxiomC{$A \Rightarrow A$}
		\AxiomC{$B \Rightarrow B$}
		\LeftLabel{\tiny{$\rightarrow L$}}
		\BinaryInfC{$A, A \rightarrow B \Rightarrow B$}
		\LeftLabel{\tiny{$\Box_{4_h}$}, $\Box_{S_h}$}
		\UnaryInfC{$\Box_n A, \Box_n (A \rightarrow B) \Rightarrow \Box_n B$}
		\LeftLabel{\tiny{$\rightarrow R$}}
		\UnaryInfC{$\Box_n (A \rightarrow B) \Rightarrow \Box_n A \rightarrow \Box_n B$}
		\LeftLabel{\tiny{$\rightarrow R$}}
		\UnaryInfC{$\Rightarrow \Box_n (A \rightarrow B) \rightarrow (\Box_n A \rightarrow \Box_n B)$}
		\DisplayProof
		\\[2 ex]
		\end{tabular}
\end{center}
\begin{flushleft}
Two different proofs of the axiom $\mathbf{H}$ in the systems $G(\mathbf{K4}_h)$ and $G(\mathbf{S4}_h)$:
\end{flushleft}
\begin{center}	
		\begin{tabular}{c c}
		\AxiomC{$ A \Rightarrow A $}
		\LeftLabel{\tiny{$wL$}}
		\UnaryInfC{$ A, \Box_n A \Rightarrow A $}
		\LeftLabel{\tiny{$\Box_{4_h} R$}}
		\UnaryInfC{$ \Box_n A \Rightarrow \Box_{n+1} A$}
		\LeftLabel{\tiny{$\rightarrow R$}}
		\UnaryInfC{$\Rightarrow \Box_n A \rightarrow \Box_{n+1} A$}
		\DisplayProof
		&
		\AxiomC{$ A \Rightarrow A $}
		\LeftLabel{\tiny{$\Box_hL$}}
		\UnaryInfC{$ \Box_n A \Rightarrow A $}
		\LeftLabel{\tiny{$\Box_{S_h}R$}}
		\UnaryInfC{$ \Box_n A \Rightarrow \Box_{n+1} A$}
		\LeftLabel{\tiny{$\rightarrow R$}}
		\UnaryInfC{$\Rightarrow \Box_n A \rightarrow \Box_{n+1} A$}
		\DisplayProof
		\\[2 ex]
		\end{tabular}
\end{center}
\begin{flushleft}
The proof of the axiom $\mathbf{4}_h$ in the systems $G(\mathbf{K4}_h)$ and $G(\mathbf{S4}_h)$:
\end{flushleft}
\begin{center}		
		\begin{tabular}{c c}
		\AxiomC{$\Box_n A \Rightarrow \Box_n A $}
		\LeftLabel{\tiny{$wL$}}
		\UnaryInfC{$ A, \Box_{n} A \Rightarrow \Box_n A $}
		\LeftLabel{\tiny{$\Box_{4_h} R$}}
		\UnaryInfC{$ \Box_n A \Rightarrow \Box_{n+1} \Box_n A$}
		\LeftLabel{\tiny{$\rightarrow R$}}
		\UnaryInfC{$\Rightarrow \Box_n A \rightarrow \Box_{n+1} \Box_n A$}
		\DisplayProof
		&
		\AxiomC{$ \Box_n A \Rightarrow \Box_n A $}
		\LeftLabel{\tiny{$\Box_{S_hR}$}}
		\UnaryInfC{$ \Box_n A \Rightarrow \Box_{n+1} \Box_n A$}
		\LeftLabel{\tiny{$\rightarrow R$}}
		\UnaryInfC{$\Rightarrow \Box_n A \rightarrow \Box_{n+1} \Box_n A$}
		\DisplayProof
		\\[2 ex]
		\end{tabular}
\end{center}
\begin{flushleft}
The proof of the axioms $\mathbf{D}_h$ and $\mathbf{T}_h$ in the systems $G(\mathbf{KD4}_h)$ and $G(\mathbf{S4}_h)$, respectively:
\end{flushleft}
\begin{center}		
		\begin{tabular}{c c}
		\AxiomC{$\bot \Rightarrow $}
		\LeftLabel{\tiny{$wR$}}
		\UnaryInfC{$\Box_n \bot, \bot \Rightarrow $}
		\LeftLabel{\tiny{$\Box_{D_h} L$}}
		\UnaryInfC{$\Box_n \bot \Rightarrow $}
		\LeftLabel{\tiny{$wL$}}
		\UnaryInfC{$\Box_n \bot \Rightarrow \bot$}
		\LeftLabel{\tiny{$\rightarrow R$}}
		\UnaryInfC{$\Rightarrow \Box_n \bot \rightarrow \bot$}
		\DisplayProof
  	    &
		\AxiomC{$A \Rightarrow A$}
		\LeftLabel{\tiny{$\Box_hL$}}
		\UnaryInfC{$\Box_n A \Rightarrow A$}
		\LeftLabel{\tiny{$\rightarrow R$}}
		\UnaryInfC{$\Rightarrow \Box_n A \rightarrow A$}
		\DisplayProof
		\end{tabular}
		
\end{center}
For the rules, the modus ponens case and the necessitation rule will be easily handled by the cut rule and the rule $\Box_{4_h}R$ or $\Box_{S_h}R$, respectively. To prove the converse, it is enough to prove that the modal rules can be simulated in the corresponding Hilbert style systems. In a more precise word, we want to show that if $\Gamma \Rightarrow \Delta$ is provable in the sequent calculus, then $\bigwedge \Gamma \rightarrow \bigvee \Delta$ is provable in its Hilbert style counterpart. To achieve this goal, it is enough to show that the system $\mathbf{K4}_h$ admits the rule $\Box_{4_h}R$, $\mathbf{KD4}_h$ admits the rules $\Box_{4_h}R$ and $\Box_{D_h}R$ and finally $\mathbf{S4}_h$ admits the rules $\Box_{S_h}R$ and $\Box_hL$.\\

1. The case of $\mathbf{K4}_h$. If we have $ \{\sigma_r\}_{r \in R}, \{ \gamma_i, \Box_{n_i} \gamma_i\}_{i \in I} \Rightarrow A$, then by IH,
\[
\mathbf{K4}_h \vdash \bigwedge \{\sigma_r\}_{r \in R}, \{ \gamma_i, \Box_{n_i} \gamma_i\}_{i \in I} \rightarrow A.
\]
Then since $n_i<n$, by a combination of the necessitation rule and some use of the axiom $\mathbf{K}_h$ 
\[
\mathbf{K4}_h \vdash \Box_n \bigwedge \{\sigma_r\}_{r \in R}, \{ \gamma_i, \Box_{n_i} \gamma_i\}_{i \in I} \rightarrow \Box_n A.
\]
Using the axiom $\mathbf{K}_h$ we can prove that $\Box_n$ commutes with conjunctions. Hence
\[
\mathbf{K4}_h \vdash  \bigwedge \{ \Box_n \sigma_r\}_{r \in R}, \{\Box_n \gamma_i, \Box_n \Box_{n_i} \gamma_i\}_{i \in I} \rightarrow \Box_n A.
\]
By $\mathbf{4}_h$ we have
\[
\mathbf{K4}_h \vdash  \bigwedge \{ \Box_n \sigma_r\}_{r \in R}, \{\Box_n \gamma_i, \Box_{n_i} \gamma_i\}_{i \in I} \rightarrow \Box_n A
\]
and by $\mathbf{H}$ and the fact that $n_i<n$ we have
\[
\mathbf{K4}_h \vdash  \bigwedge \{ \Box_n \sigma_r\}_{r \in R}, \{ \Box_{n_i} \gamma_i\}_{i \in I} \rightarrow \Box_n A.
\]

2. The case of $\mathbf{KD4}_h$. Showing that $\mathbf{KD4}_h$ admits the rule $\Box_{4_h}$ is similar to the case 1. To show that the system admits the rule $\Box_{D_h}R$, it is enough to replace $A$ in the case 1 by $\bot$. We have
\[
\mathbf{KD4}_h \vdash  \bigwedge \{ \Box_n \sigma_r\}_{r \in R}, \{ \Box_{n_i} \gamma_i\}_{i \in I} \rightarrow \Box_n \bot
\]
but we know that $\mathbf{KD4}_h \vdash \neg \Box_n \bot$. Therefore,
 \[
\mathbf{KD4}_h \vdash  \bigwedge \{ \Box_n \sigma_r\}_{r \in R}, \{ \Box_{n_i} \gamma_i\}_{i \in I} \rightarrow \bot.
\]

3. The case of $\mathbf{S4}_h$. The proof for the rule $\Box_{S_h}R$ is similar to the case 1. For the rule $\Box_hL$, if we have {$ \Gamma, A \Rightarrow \Delta$}, then by IH,
\[
\mathbf{S4}_h \vdash A \wedge \bigwedge \Gamma \rightarrow \bigvee \Delta.
\]
Since we know $\mathbf{S4}_h \vdash \Box_n A \rightarrow A$, hence
\[
\mathbf{S4}_h \vdash \Box_n A \wedge \bigwedge \Gamma \rightarrow \bigvee \Delta.
\] 
\end{proof}
\begin{thm}(Cut Elimination)\label{t2-8}
The Systems $G(\mathbf{K4}_h)$, $G(\mathbf{KD4}_h)$ and $G(\mathbf{S4}_h)$ have the cut elimination property.
\end{thm}
\begin{proof}
The proof is similar to the usual proof of cut elimination in the systems $G(\mathbf{K4})$, $G(\mathbf{KD4})$ and $G(\mathbf{S4})$. The strategy is proving the following principal lemma for the systems $G(\mathbf{K4}_h)$, $G(\mathbf{KD4}_h)$ and $G(\mathbf{S4}_h)$. Define the complexity of a formula $A$, as its length. Then by a $d$-proof we mean a proof whose cut formulas have the complexity strictly less than $d$. We claim:\\

\textbf{Principal Lemma.} Let $A$ be a formula of the complexity $d$ and $\pi$ and $\pi'$ be some $d$-proofs of $\Gamma \Rightarrow \Delta, A$ and $\Gamma', A \Rightarrow \Delta'$, respectively. Then there is a $d$-proof $\pi''$ of $\Gamma, \Gamma'-A \Rightarrow \Delta-A, \Delta'$ in which $\Gamma'-A$ means the multiset $\Gamma'$ after eliminating all occurrences of $A$. The same holds for $\Delta -A$.\\

Since the proof of the principal lemma for the systems $G(\mathbf{K4}_h)$, $G(\mathbf{KD4}_h)$ and $G(\mathbf{S4}_h)$ are more or less the same, we will prove it for $G(\mathbf{KD4}_h)$. The proof as usual is by induction on the addition of the height of $\pi$ and the height of $\pi'$. The cases in which the last rule of $\pi$ or $\pi'$ is an axiom, structural or propositional rule are easy to check. Therefore, we will focus on the case that the last rule of both $\pi$ and $\pi'$ are modal rules. Since the right side of $\Gamma \Rightarrow \Delta, A$ is not empty, the last rule of $\pi$ should be $\Box_{4_h} R$. But the last rule of $\pi'$ could be both of the rules $\Box_{4_h} R$ or $\Box_{D_h} R$. Assume that it is $\Box_{4_h} R$. There are two cases. The first one is when $A$ is one of the boxed formulas in the premise of the last rule of $\pi'$. The other case is when $A$ is between un-boxed formulas in the premise sequent. For the first case we have the following combination:\\

\begin{center}
\begin{tabular}{c c}
		\AxiomC{$\{\sigma_r\}_{r \in R}, \{ \gamma_i, \Box_{n_i} \gamma_i\}_{i \in I} \Rightarrow A$}
		\UnaryInfC{$\{\Box_n \sigma_r\}_{r \in R}, \{ \Box_{n_i} \gamma_i\}_{i \in I} \Rightarrow \Box_n A$}
		\DisplayProof
		&
        \AxiomC{$\{\sigma'_s\}_{s \in S}, \{ \gamma'_j, \Box_{m_j} \gamma'_j\}_{j \in J}, A, \Box_n A \Rightarrow B$}
		\UnaryInfC{$\{\Box_k \sigma'_s\}_{s \in S}, \{ \Box_{m_j} \gamma'_j\}_{j \in J}, \Box_n A \Rightarrow \Box_k B$}
		\DisplayProof
\end{tabular}
\end{center}

Consider the proof $\pi$ of
\[
\{\Box_n \sigma_r\}_{r \in R}, \{ \Box_{n_i} \gamma_i\}_{i \in I} \Rightarrow \Box_n A
\]
and the subproof of $\pi'$ that leads to 
\[
\{\sigma'_s\}_{s \in S}, \{ \gamma'_j, \Box_{m_j} \gamma'_j\}_{j \in J}, A, \Box_n A \Rightarrow B.
\]
By IH, we can construct a $d$-proof of
\[
(\{\sigma'_s\}_{s \in S}, \{ \gamma'_j, \Box_{m_j} \gamma'_j\}_{j \in J}, A)-\Box_n A, \{\Box_n \sigma_r\}_{r \in R}, \{ \Box_{n_i} \gamma_i\}_{i \in I} \Rightarrow B.
\]
Then use cut on $A$ for the following two sequents:
\[
\{\sigma_r\}_{r \in R}, \{ \gamma_i, \Box_{n_i} \gamma_i\}_{i \in I} \Rightarrow A
\]
and
\[
(\{\sigma'_s\}_{s \in S}, \{ \gamma'_j, \Box_{m_j} \gamma'_j\}_{j \in J}, A)-\Box_n A, \{\Box_n \sigma_r\}_{r \in R}, \{ \Box_{n_i} \gamma_i\}_{i \in I} \Rightarrow B
\]
to have
\[
(\{\sigma'_s\}_{s \in S}, \{ \gamma'_j, \Box_{m_j} \gamma'_j\}_{j \in J})-\{\Box_n A, A\}, \{\sigma_r\}_{r \in R},  \{\gamma_i, \Box_{n_i} \gamma_i\}_{i \in I}, \{\Box_n \sigma_r\}_{r \in R}, \{ \Box_{n_i} \gamma_i\}_{i \in I} \Rightarrow B
\]
and then use $\Box_{4_h} R$ to have
\[
\{\Box_n \sigma_r\}_{r \in R}, \{\Box_k \sigma_s\}_{s \in S}, \{ \Box_{n_i} \gamma_i\}_{i \in I}, (\{ \Box_{m_j} \gamma'_j\}_{j \in J})-\Box_n A \Rightarrow \Box_k B.
\]
The important thing is that since the rule $\Box_{4_h} R$ is applied in $\pi$ and $\pi'$ we have to have $n_i < n$ and $n, m_j<k$, and therefore $n, n_i, m_j<k$, which guarantees the application of the last $\Box_{4_h} R$ rule in the above tree. Notice that this proof eliminates all the occurrences of $\Box_n A$ in $\{\Box_k \sigma'_s\}_{s \in S}, \{ \Box_{m_j} \gamma'_j\}_{j \in J}$. The reason is that $n<k$ and hence there is not any occurrence of $\Box_n A$ in $\{\Box_k \sigma'_s\}_{s \in S}$. On the other hand, the proof eliminates all the occurrences of $\Box_n A$ in $ \{ \Box_{m_j} \gamma'_j\}_{j \in J}$, which completes the claim. \\

For the second case we have the following combination:

\begin{center}
\begin{tabular}{c c}
		\AxiomC{$\{\sigma_r\}_{r \in R}, \{ \gamma_i, \Box_{n_i} \gamma_i\}_{i \in I} \Rightarrow A$}
		\UnaryInfC{$\{\Box_k \sigma_r\}_{r \in R}, \{ \Box_{n_i} \gamma_i\}_{i \in I} \Rightarrow \Box_k A$}
		\DisplayProof
		\\[4 ex]
\end{tabular}
\begin{tabular}{c}	
		\AxiomC{$\{\sigma'_s\}_{s \in S}, A, A, \ldots, A, \{ \gamma'_j, \Box_{m_j} \gamma'_j\}_{j \in J} \Rightarrow B$}
		\UnaryInfC{$\{\Box_k \sigma'_s\}_{s \in S}, \Box_k A, \Box_k A, \ldots, \Box_k A, \{ \Box_{m_j} \gamma'_j\}_{j \in J} \Rightarrow \Box_k B$}
		\DisplayProof
		\end{tabular}
\end{center}

Then consider the subproofs of $\pi$ and $\pi'$ leading to the sequents
\[ 
\{\sigma_r\}_{r \in R}, \{ \gamma_i, \Box_{n_i} \gamma_i\}_{i \in I} \Rightarrow A
\]
and 
\[
\{\sigma'_s\}_{s \in S}, A, A, \ldots, A, \{ \gamma'_j, \Box_{m_j} \gamma'_j\}_{j \in J} \Rightarrow B.
\]
By some cuts on $A$, some contractions and applying the rule $\Box_{4_h} R$, we have 

\begin{center}
\begin{tabular}{c c}
		\AxiomC{$\{\sigma_r\}_{r \in R}, \{ \gamma_i, \Box_{n_i} \gamma_i\}_{i \in I} \Rightarrow A$}
		\AxiomC{ $\{\sigma'_s\}_{s \in S}, A, A, \ldots, A, \{ \gamma'_j, \Box_{m_j} \gamma'_j\}_{j \in J} \Rightarrow B$ }
		\BinaryInfC{$\{\sigma'_s\}_{s \in S}, \{ \gamma'_j, \Box_{m_j} \gamma'_j\}_{j \in J}, \{\sigma_r\}_{r \in R},  \{\gamma_i, \Box_{n_i} \gamma_i\}_{i \in I} \Rightarrow B$}
		\UnaryInfC{$\{\Box_k \sigma_r\}_{r \in R}, \{\Box_k \sigma'_s\}_{s \in S}, \{ \Box_{n_i} \gamma_i\}_{i \in I}, \{ \Box_{m_j} \gamma'_j\}_{j \in J} \Rightarrow \Box_k B$}
		\DisplayProof
		\end{tabular}
\end{center}

Again, since $n_i<k$ and $m_j<k$ we can apply the last rule.\\

Finally, if the last rule of $\pi'$ is $\Box_D R$, then it is enough to erase all the occurrences of $B$ and $\Box_k B$'s in the above proof. The proof would be exactly the same.
\end{proof}
The other ingredient that we will need, is a strong version of the necessitation rule.
\begin{thm}(Strong Necessitation)\label{t2-9}
\begin{itemize}
\item[$(i)$]
If $\bigwedge_{i=1}^{k}\Box_{m_i} A_i \vdash_{\mathbf{K4}_h} A$ and for any $i$, $m_i \leq n$, then $\bigwedge_{i=1}^{k}\Box_{m_i} A_i \vdash_{\mathbf{K4}_h} \Box_n A$.
\item[$(ii)$]
If $\bigwedge_{i=1}^{k}\Box_{m_i} A_i \vdash_{\mathbf{S4}_h} A$ and for any $i$, $m_i \leq n$, then $\bigwedge_{i=1}^{k}\Box_{m_i} A_i \vdash_{\mathbf{S4}_h} \Box_n A$.
\end{itemize}
\end{thm}
\begin{proof}
To prove the theorem, we need the following translation. Assume $n$ is a number and $Z$ is a formula such that $r(Z)<n$. Define $A^Z$ inductively as follows:
\begin{itemize}
\item[$(i)$]
If $B$ is an atom, $B^Z=B$.
\item[$(ii)$]
$(B \circ C)^Z=B^Z \circ C^Z$ for all $\circ \in \{\wedge, \vee, \rightarrow \}$
\item[$(iii)$]
$(\neg B)^Z=\neg B^Z$.
\item[$(iv)$]
If $i<n$, $(\Box_i B)^Z=\Box_i B^Z$ and if $i \geq n$ then $(\Box_i B)^Z=\Box_i (Z \rightarrow B^Z)$.
\end{itemize} 
We have the following claim.\\

\textbf{Claim.} If $\mathbf{K4}_h \vdash A$ then $ \mathbf{K4}_h \vdash A^Z$. And if $\mathbf{S4}_h \vdash A$ then $Z \vdash_{\mathbf{S4}_h} A^Z$.\\

We will just prove the case for $\mathbf{K4}_h$. The case for $\mathbf{S4}_h$ is similar. First, observe that for any formula $C$ if $n>r(C)$, then $C^Z=C$. And secondly, $r(C^Z)=r(C)$ which is easy to prove by induction on the complexity of $C$.\\
To prove the claim for $\mathbf{K4}_h$ we will use induction on the length of the proof of $A$. The case for axioms and the case of the modus ponens rule are easy to check. If $A$ is a result of the necessitation rule then we have $A=\Box_m B$ and $\mathbf{K4}_h \vdash B$. If $m<n$, then by the first observation $A^Z=A$. Therefore, the claim is trivial. If $m \geq n$ then by IH we have $ \vdash_{\mathbf{K4}_h} B^Z$. Hence, $\mathbf{K4} \vdash Z \rightarrow B^Z$. Since $r(Z)<n \leq m$ and $r(B^Z)=r(B)<m$, by necessitation we have $\mathbf{K4}_h \vdash \Box_m ( Z \rightarrow B^Z)$. Hence, $\mathbf{K4}_h \vdash (\Box_m B)^Z$. \qed \\

To prove the theorem, assume that $\bigwedge_{i=1}^{k}\Box_{m_i} B_i \wedge \bigwedge_{j=1}^{r} \Box_n C_j \vdash_{\mathbf{K4}_h} A$ where $m_i<n$. Then $\bigwedge_{j=1}^{r} \Box_n C_j \vdash_{\mathbf{K4}_h} \bigwedge_{i=1}^{k}\Box_{m_i} B_i \rightarrow A$. Pick $Z=\bigwedge _{j=1}^{r} C_j$. Since $r(Z) < n $, we have
\[
\mathbf{K4} \vdash (\bigwedge _{j=1}^{r} \Box_n C_j)^Z \rightarrow (\bigwedge_{i=1}^{k}\Box_{m_i} B_i \rightarrow A)^Z.
\]
Since $r(\bigwedge_{i=1}^{k}\Box_{m_i} B_i \rightarrow A)<n$ we have $(\bigwedge_{i=1}^{k}\Box_{m_i} B_i \rightarrow A)^Z=\bigwedge_{i=1}^{k}\Box_{m_i} B_i \rightarrow A$. Moreover, because we have $(\Box_n C_j)^Z=\Box_n (\bigwedge_{j=1}^{r} C_j \rightarrow C_j)$, then $\mathbf{K4}_h \vdash \bigwedge_{i=1}^{k}\Box_{m_i} B_i \rightarrow A$. Hence, by necessitation and the axiom $\mathbf{K}_h$ we have 
\[
\mathbf{K4}_h \vdash \bigwedge_{i=1}^{k}\Box_n \Box_{m_i} B_i \rightarrow \Box_n A.
\]
Then by axiom $\mathbf{4}_h$ we have 
\[
\bigwedge_{i=1}^{k}\Box_{m_i} B_i \wedge \bigwedge_{j=1}^{r} \Box_n C_j \vdash_{\mathbf{K4}_h} \Box_n A.
\]

For $\mathbf{S4}_h$, the proof is similar, but the only difference is that by the claim we will have $\mathbf{S4}_h \vdash \bigwedge_{j=0}^{r} C_j \wedge \bigwedge_{i=1}^{k}\Box_{m_i} B_i \rightarrow A$. Hence by the same considerations as above we have
\[
\bigwedge_{i=1}^{k}\Box_{m_i} B_i \wedge \bigwedge_{j=1}^{r} \Box_n C_j \vdash_{\mathbf{S4}_h} \Box_n A.
\]
\end{proof}
We stated that our strategy to prove the completeness of the hierarchical modal logics is reducing their completeness to the completeness of the usual modal theories. To achieve this goal we need the following translation and also its completeness.
\begin{dfn}\label{t2-10}
Let $Q=\{q_n\}_{n=0}^{\infty}$ be a set of new atomic variables which are not occurred in the formulas in $\mathcal{L}_{\infty}$. Then define the translation $t: \mathcal{L}_{\infty} \to \mathcal{L}_{\Box}(Q)$ as follows:
\begin{itemize}
\item[$(i)$]
If $A$ is an atom, $A^t=A$.
\item[$(ii)$]
$(A \circ B)^t=A^t \circ B^t$ for all $\circ \in \{\wedge, \vee, \rightarrow \}$
\item[$(iii)$]
$(\neg A)^t=\neg A^t$.
\item[$(iv)$]
$(\Box_n A)^t=\Box (\bigwedge_{i=0}^{n}q_i \rightarrow A^t)$.
\end{itemize} 
\end{dfn}
We will prove the following completeness of the translation $t$ between $\mathbf{K4}(Q)$ and $\mathbf{K4}_h$ and also between $\mathbf{S4}(Q)$ and $\mathbf{S4}_h$.
\begin{thm}\label{t2-11}
\begin{itemize}
\item[$(i)$]
If $\Gamma^t \vdash_{\mathbf{K4}(Q)} A^t$, then $\Gamma \vdash_{\mathbf{K4}_h} A$.
\item[$(ii)$]
If $\Gamma^t \vdash_{\mathbf{S4}(Q)} A^t$, then $\Gamma \vdash_{\mathbf{S4}_h} A$.
\end{itemize}
\end{thm}
\begin{proof}
First of all, it is obvious that it is enough to prove the theorem for $\Gamma=\emptyset$. The reason is that in both sides, proofs just use a finite subset of $\Gamma$ and $\Gamma^t$. Secondly, we will present a proof for both of the logics $\mathbf{K4}(Q)$ and $\mathbf{S4}(Q)$, simultaneously. To achieve this goal we will use $L$ to denote these logics, $L_h$ to denote their hierarchical counterparts and $G(L)$ to denote $L$'s sequent calculus as introduced in the Preliminaries.\\

Assume $L \vdash A^t$. Hence, we have $G(L) \vdash \; \Rightarrow A^t$. Therefore, there exists a cut-free proof $\pi$ of $\Rightarrow A^t$ in which all formulas are sub-formulas of $A^t$. Define the set $X$ of usual modal formulas as the least set such that:
\begin{itemize}
\item[$(i)$]
If $B$ is an atom, $B \in X$. Moreover, $\bot, \top \in X$.
\item[$(ii)$]
If $B, C \in X$ then $B \circ C \in X$ for all  $\circ \in \{\wedge, \vee, \rightarrow \}$
\item[$(iii)$]
If $B \in X$ then $\neg B \in X$.
\item[$(iv)$]
If $B \in X$ then:\\
1. $\Box (\bigwedge_{i=0}^{n}q_i \rightarrow B) \in X$ if $n$ is greater than all indices of $q_j$'s occurring in $B$. Formulas defined in this case are called the \textit{first kind} boxed formulas in $X$.\\
2. $\Box (\bigwedge_{i=0}^{m}q_i \wedge \bigwedge_{i=m+1}^{n} \bot \rightarrow B) \in X$ if $m$ is greater than or \textit{equal to} all indices of $q_j$'s occurring in $B$ and $m<n$.  Formulas defined in this case are called the \textit{second kind} boxed formulas in $X$.
\end{itemize} 
It is easy to check that $X$ includes all sub-formulas of $A^t$. The reason is that all boxed sub-formulas of $A^t$ have the form $\Box(\bigwedge_{i=0}^{n}q_i \rightarrow B)$ in which $n$ is greater than all $q_k$'s occurring in $B$. Therefore, every formula in the proof $\pi$ belongs to $X$. To prove the theorem, we need some more ingredients. The first is the notion of $X$-proof. An $X$-proof is a cut-free proof in the system $G(L)$ in which all formulas belong to $X$. For instance, $\pi$ is an $X$-proof. The second notion is the rank of formulas in $X$. If $B \in X$, by $r(B)$ we mean the greatest number $n$ such that $q_n$ occurred in the formula $B$. The third notion is a good $X$-proof. An $X$-proof is called good if for any occurrence of the rules 

\begin{center}
  	\begin{tabular}{c c}
		\AxiomC{$ \Gamma, \Box \Gamma \Rightarrow B$}
		\LeftLabel{\tiny{$\Box_4 R$}}
		\UnaryInfC{$\Box \Gamma \Rightarrow \Box B$}
		\DisplayProof
		&
		\AxiomC{$ \Box \Gamma \Rightarrow B$}
		\LeftLabel{\tiny{$\Box_S R$}}
		\UnaryInfC{$\Box \Gamma \Rightarrow \Box B$}
		\DisplayProof
		\end{tabular}
\end{center}
the rank of $\Box B$ is bigger than or equal to the maximum of the ranks of formulas in $\Box \Gamma$.\\

\textbf{Claim 1.} Let $n$ be a natural number and $\sigma_n$ be the substitution which sends all $q_i$ for $i > n$ to $\bot$. Then if $\alpha$ is a good $X$-proof, $\alpha(\sigma_n)$ is so.\\

First of all, notice that the set $X$ is closed under these kinds of substitutions. The proof is just by structural induction on $X$. The important case is the box case. For the first kind, $B=\Box (\bigwedge_{i=0}^{m}q_i \rightarrow C)$. If $n \geq m$, then $B(\sigma_n)=B$, because there is no $q_i$ in $B$ such that $i>n$. Therefore, $B(\sigma_n) \in X$. If $n<m$, then $B(\sigma_n)=\Box (\bigwedge_{i=0}^{n}q_i \wedge \bigwedge_{i=n+1}^{m} \bot \rightarrow C(\sigma_n))$. By IH, $C(\sigma_n) \in X$. Since $n+1 \leq m$, $B(\sigma_n)$ is of the second kind. We know that $n$ is the biggest number between the indices occurring in $C(\sigma_n)$. The reason is that we substitute all $q$'s with greater index by $\bot$. Hence, $B(\sigma_n) \in X$.\\  
For the second kind, $B=\bigwedge_{i=0}^{m}q_i \wedge \bigwedge_{i=m+1}^{k} \bot \rightarrow B$. If $n \geq m$, $B(\sigma_n)=B$ because there is no $q_i$ in $B$ such that $i>n$. Therefore, $B(\sigma_n) \in X$. If $n<m$, then $B(\sigma_n)=\Box (\bigwedge_{i=0}^{n}q_i \wedge \bigwedge_{i=n+1}^{k} \bot \rightarrow C(\sigma_n))$. Since $n+1 \leq m$, $B(\sigma_n)$ is of the second kind. We know that $n$ is the biggest number between the indices occurring in $C(\sigma_n)$. The reason is that we substitute all $q$'s with greater index by $\bot$. Hence, $B(\sigma_n) \in X$.\\ 

Secondly, it is easy to see that if $B$ is a boxed formula in $X$ then $r(B(\sigma_n))=min\{r(B), n\}$. It is enough to check all four above possibilities. For two of them $B(\sigma_n)=B$ and hence the claim is obvious. For the other two, both ranks are $n$ and $m>n$.\\
We have shown that $X$ is closed under the substitution $\sigma_n$ which means that if $\alpha$ is an $X$-proof then $\alpha(\sigma_n)$ is an $X$-proof as well. We want to show that it is also a good $X$-proof. The reason is that if we apply the rule $\Box R$ in $\alpha$, then since $\alpha$ is a good $X$-proof, $r(\Box B) \geq r(\Box \Gamma)$. After substitution, we have $r(\Box B(\sigma_n))=min\{r(B), n\} \geq min\{r(\Box \Gamma), n\}=r(\Box \Gamma (\sigma_n))$. Therefore, $\alpha(\sigma)$ is also a good $X$-proof. This completes the proof of the Claim 1.\\ 

\textbf{Claim 2.} If $\Gamma \Rightarrow \Delta$ has an $X$-proof, then there exists a set $\Box \Sigma$ of formulas of the second kind in $X$ such that $\Box \Sigma, \Gamma \Rightarrow \Delta$ has a good $X$-proof. \\

To prove this claim, we use induction on the length of the proof of $\Gamma \Rightarrow \Delta$. If the last rule is an axiom, a structural rule or a propositional rule, then the claim is obvious from the IH. For the modal rules:\\

1. If $L=\mathbf{K4}(Q)$ and the last rule is $\Box_4 R$. Then we have $\Box \Gamma \Rightarrow \Box A$ proved by $\Gamma, \Box \Gamma \Rightarrow A$. By IH, there exists $\Box \Sigma \subseteq X$ such that $\Box \Sigma, \Gamma, \Box \Gamma \Rightarrow A$ has a good $X$-proof. Divide $\Gamma$ in two parts, $\Gamma_0$ and $\Gamma_1$ such that $r(\Gamma_0) \leq r(A)$ and $r(\gamma) > r(A)$ for all $\gamma \in \Gamma_1$. We have a good $X$-proof for 
\[
\Box \Sigma, \Gamma_0, \Box \Gamma_0, \Gamma_1, \Box \Gamma_1 \Rightarrow A.
\]
Use $\sigma_{r(A)}$ to substitute all $q_j$'s, $j > r(A)$, by $\bot$. Since there is not any occurrence of these $q_j$'s in $\Gamma_0$ and $A$, they do not change after applying the substitution. Hence we have
\[
\Box \Sigma', \Gamma_0, \Box \Gamma_0, \Gamma_1', \Box \Gamma_1' \Rightarrow A
\]
in which $Y'$ means $Y$ after applying the substitution. By Claim 1 we know that the new proof is also a good $X$-proof. Use left weakening for $\Sigma'$ and then by $\Box_4 R$ we have 
\[
\Box \Sigma', \Box \Gamma_0, \Box \Gamma_1' \Rightarrow \Box A.
\]
Again by right weakening for $\Box \Gamma_1$ we have
\[
\Box \Sigma', \Box \Gamma_0, \Box \Gamma_1', \Box \Gamma_1 \Rightarrow \Box A.
\]
Therefore
\[
\Box \Sigma', \Box \Gamma_1', \Box \Gamma \Rightarrow \Box A
\]
which is what we wanted. Notice that the use of $\Box_4 R$ is now a good one, because $r(A) \geq r(\Box \Sigma', \Box \Gamma_0, \Box \Gamma_1')$. Therefore, the new proof is a good $X$-proof. Moreover, notice that the formulas in $\Box \Sigma' \cup \Box \Gamma'_1$ are of the second kind. The reason is that on the one hand $\Box \Sigma'$ is $(\Box \Sigma)(\sigma_{r(A)})$ and formulas in $\Box \Sigma$ are of the second kind by IH, hence the formulas after substitution should also be of the second kind. On the other hand, for any formula $\gamma \in \Gamma_1$, $r(\gamma)>r(A)$. Therefore, if $\gamma=\bigwedge_{i=0}^{k} r_i \rightarrow \beta$ then $q_{r(\gamma)}$ is among $r_i$'s and after substitution, it changes to $\bot$ which means that $\gamma(\sigma_{r(A)})$ is of the second kind.\\

2. If $L=\mathbf{S4}(Q)$. If the last rule is $\Box_S R$ the proof is similar to the the case $L=\mathbf{K4}(Q)$. If the last rule is $\Box L$ then the proof is straightforward.\\

This completes the proof of the Claim 2.\\

Define a translation $s: X \to \mathcal{L}_{\infty}$ as follows: 
\begin{itemize}
\item[$(i)$]
$B$ is an atom: If $B \notin Q$, $B^s=B$. If $B \in Q$, $B^s=\top$. If $B=\bot$ then $B^s=\bot$ and if $B=\top$ then $B^s=\top$.
\item[$(ii)$]
$(B \circ C)^s=B^s \circ C^s$ for all  $\circ \in \{\wedge, \vee, \rightarrow \}$.
\item[$(iii)$]
$(\neg B)^s=\neg B^s$.
\item[$(iv)$]
1. If $B= \Box (\bigwedge_{i=0}^{n}q_i \rightarrow C)$ then $B^s=\Box_n C^s$.\\
2. If $B= \Box (\bigwedge_{i=0}^{m}q_i \wedge \bigwedge_{i=m+1}^{n} \bot \rightarrow C)$ then $B^s=\top$.
\end{itemize} 
This translation is a left converse for the translation $t$. In fact, for any $A \in \mathcal{L}_{\infty}$, $(A^t)^s=A$. It is a consequence of an easy induction on $A$. We will prove some properties for this translation. First of all, we have to show that the image of any formula in $X$ belongs to $\mathcal{L}_{\infty}$. The proof is by structural induction on $X$. The important case is the first case of the boxed case. If $B$ constructed from the first case of the definition of $X$ then $n$ is greater than all $q$'s indices occurring in $C$. On the other hand, all box $\Box_m$ introduced in $C^s$ are introduced from the first case as well and hence $m$ can not be greater than the biggest index of $q$'s in $C$. Therefore, $n$ is greater than all $m$'s, and by IH we know that $C^s \in \mathcal{L}_{\infty}$, hence $B^s \in \mathcal{L}_{\infty}$.\\
Secondly, we want to show that for all formulas $B \in X$, $r(B^s) \leq r(B)$. The proof is by structural induction on $X$. Again the important case is the first case of the box case. In that case, we have $B=\Box (\bigwedge_{i=0}^{n}q_i \rightarrow C)$ then $B^s=\Box_n C^s$. And hence $r(B)=n$. Moreover, $B^s=\Box_n C^s$. therefore, $r(B^s)=n$ which completes the proof.\\
Thirdly, notice that if $\Box B \in X$ is of the second kind, then $B^s$ is provably equivalent to $\top$ in $L_h$. The reason is that if $\Box B$ is of the second kind, then $B=\bigwedge_{i=0}^{m}q_i \wedge \bigwedge_{i=m+1}^{n} \bot \rightarrow C$ therefore, $B$ has a $\bot$ in its premises, which means that $B^s$ is equivalent to $\top$.\\    

\textbf{Claim 3.} If there is a good $X$-proof for $\Gamma \Rightarrow \Delta$, then $L_h \vdash \bigwedge \Gamma^s \rightarrow \bigvee \Delta^s $.\\

The proof is based on induction on the length of the $X$-proof. If the last rule is an axiom, a structural rule or a propositional rule, then the claim follows from the IH. The reason is that $s$ commutes with the propositional connectives and $L_h$ proves all propositional tautologies. For the modal rules, we have the following two cases:\\

1. If $L=\mathbf{K4}(Q)$ and if the last rule is the modal rule $\Box_4 R$, then by IH, we know $(\Box \Gamma)^s, \Gamma^s \vdash_{\mathbf{K4}_h} A^s $. If $\Box A \in X$ is of the second kind, then by definition $(\Box A)^s=\top$ and therefore there is nothing to prove. If it is of the first kind, then $r((\Box A)^s)=r(A)$. Since the proof is a good $X$-proof, $r(A) \geq r(\Box \Gamma)$. Furthermore, we know that $r((\Box \Gamma)^s) \leq r(\Box \Gamma)$. Therefore, $r((\Box A)^s) \geq r((\Box \Gamma)^s) $.\\
On the other hand, for any $\Box \gamma \in \Box \Gamma \subseteq X$, if the formula is of the second kind, then $\gamma^s$ is equivalent to $\top$ provably in $L$. Moreover, $(\Box \gamma)^s=\top$ by definition. Therefore, we can ignore this kind of boxed formulas in $\Box \Gamma$ and w.l.o.g. we can assume that all formulas in $\Box \Gamma$ are of the first kind. Therefore, $\gamma_i=\bigwedge_{i=0}^{n_i}q_i \rightarrow \beta_i$, hence $\gamma_i^s=\bigwedge_{i=0}^{n_i} \top \rightarrow \beta_i^s$ which means that $\gamma_i^s$ and $\beta_i^s$ are provably equivalent in $\mathbf{K4}_h$. If $A=\bigwedge_{i=0}^{m}q_i \rightarrow B$, by IH we know that $(\Box \Gamma)^s, \Gamma^s \vdash_{\mathbf{K4}_h} A^s$ and therefore, $\Box_{n_i} \beta_i^s, \beta_i^s \vdash_{\mathbf{K4}_h} B^s$. Since $m=r(A) \geq n_i=r(\gamma_i)$ for any $i$, by strong necessitation Theorem \ref{t2-9}, we have 
\[
\Box_{n_i} \beta_i^s  \vdash_{\mathbf{K4}_h} \Box_m(\bigwedge \beta_i^s \rightarrow B^s).
\] 
Hence
\[
\Box_{n_i} \beta_i^s, \Box_m \beta_i^s \vdash_{\mathbf{K4}_h} \Box_m B^s .
\]
Since $n_i \leq m$, we have 
\[
\Box_{n_i} \beta_i^s \vdash_{\mathbf{K4}_h} \Box_m B^s .
\]
Hence
$
(\Box \Gamma)^s \vdash_{\mathbf{K4}_h} (\Box A)^s
$ which completes the proof.\\

2. If $L=\mathbf{S4}(Q)$ and the last rule is $\Box_S R$ the proof is similar to the case 1. Consider the last rule is $\Box L$. Then $\Gamma, \Box A \Rightarrow \Delta$ is proved by $\Gamma, A \Rightarrow \Delta$. By IH, we have $\mathbf{S4}_h \vdash \bigwedge \Gamma^s \wedge A^s \rightarrow \bigvee \Delta^s$. If $\Box A$ is of the second kind, then $A^s$ is equivalent to $\top$ provably in $\mathbf{S4}_h$ and $(\Box A)^s=\top$ by definition. Therefore, there is nothing to prove. If $A$ is of the first kind, then $A=\bigwedge_{i=0}^{m}q_i \rightarrow B$. Then we know that $A^s$ is equivalent to $B^s$ provably in $\mathbf{S4}_h$. Therefore
\[
\mathbf{S4}_h \vdash \bigwedge \Gamma^s \wedge B^s \rightarrow \bigvee \Delta^s.
\]
Since $\mathbf{S4}_h \vdash \Box_m B^s \rightarrow B^s $, we have 
\[
\mathbf{S4}_h \vdash \bigwedge \Gamma^s \wedge \Box_m B^s \rightarrow \bigvee \Delta^s
\]
which is what we wanted. \qed \\

Based on the claims we have proved so far, we can prove the theorem. If $L \vdash A^t$, then we know that there is an $X$-proof of $\Rightarrow A^t$. Then by Claim 2, there is $\Box \Sigma \subseteq X$ such that $\Box \Sigma \Rightarrow A^t$ has a good $X$-proof and every formula in $\Box \Sigma$ is of the second kind. Then by Claim 3, we know that $L_h \vdash \bigwedge (\Box \Sigma)^s \rightarrow (A^t)^s$. We know that $(A^t)^s=A$. On the other hand, since the elements in $\Box \Sigma$ are of the second kind, for any $\sigma \in \Sigma$, $(\Box \sigma)^s=\top$. Therefore, we have $L_h \vdash A$. 
\end{proof}
From now on, if $A \in \mathcal{L}_{\Box}$ is a usual modal formula, and $w$ is a witness for it, by $A(w)$ we mean a formula in $\mathcal{L}_{\infty}$ substituting any occurrence of box with $\Box_n$, when $n$ is a witness for that occurrence. Moreover, by the forgetful translation $f: \mathcal{L}_{\infty} \to \mathcal{L}_{\Box}$ we mean a function which translates atomic formulas and propositional connectives to themselves and sends $\Box_n$ to $\Box$. Notice that there exists some witness $w$ for $A^f$ such that $A^f(w)=A$.
\begin{lem}\label{t6-2}
\begin{itemize}
\item[$(i)$]
For any $A \in \mathcal{L}_{\infty}$ and any natural numbers $m, n>r(A)$, $\mathbf{GL}_h \vdash \Box_{m}A \leftrightarrow \Box_{n}A$.
\item[$(ii)$]
For any $A \in \mathcal{L}_{\Box}$ and any witnesses $u$ and $v$ for $A$, $\mathbf{GL}_h \vdash A(u) \rightarrow A(v)$.
\end{itemize}
\end{lem}
\begin{proof}
For $(i)$ it is enough to show that if $n>r(A)$,
$
\mathbf{GL}_h \vdash \Box_{n}A \leftrightarrow \Box_{n+1}A
$. The part $\Box_{n}A \rightarrow \Box_{n+1}A$ is an instance of the axiom $\mathbf{H}$ and is provable in $\mathbf{GL}_h$. To prove the other part, use the axiom $\mathbf{L}_h$ for $A$. We have $\mathbf{GL}_h \vdash \Box_{n+1}(\Box_n A \rightarrow A) \rightarrow \Box_n A$. Since
\[
\mathbf{GL}_h \vdash \Box_{n+1} A \rightarrow \Box_{n+1}(\Box_n A \rightarrow A)
\]
we have $\mathbf{GL}_h \vdash \Box_{n+1}A \rightarrow \Box_n A$.\\
For $(ii)$. Use induction on $A$. The atomic and propositional cases are straightforward. For the modal case assume $A=\Box B$. Then we know that $u=(n, u')$ and $v=(m, v')$ such that $n$ is larger than all numbers in $u'$ and $m$ is larger than all numbers in $v'$. Pick $k=max\{m, n\}$. By IH, $\mathbf{GL}_h \vdash B(u') \leftrightarrow B(v')$. Therefore,
$
\mathbf{GL}_h \vdash \Box_k B(u') \leftrightarrow \Box_k B(v').
$
On the other hand by $(i)$ we have
$
\mathbf{GL}_h \vdash \Box_k B(u') \leftrightarrow \Box_n B(u')
$
and
$
\mathbf{GL}_h \vdash \Box_k B(v') \leftrightarrow \Box_m B(v').
$
Hence
$
\mathbf{GL}_h \vdash \Box_n B(u') \leftrightarrow \Box_m B(v').
$
\end{proof}
\begin{thm}\label{t6-3}
\begin{itemize}
\item[$(i)$]
If $\Gamma \vdash_{\mathbf{GL}_h} A$ then $\Gamma^f \vdash_{\mathbf{GL}} A^f$.
\item[$(ii)$]
For any set of modal formulas $\Gamma \cup \{A\} \subseteq \mathcal{L}_{\Box}$ and any witnesses $v$ and $w$ for $\Gamma$ and $A$ respectively, if $\Gamma \vdash_{\mathbf{GL}} A$ then $\Gamma(v) \vdash_{\mathbf{GL}_h} A(w)$.
\end{itemize}
\end{thm}
\begin{proof}
$(i)$ is easy to check. For $(ii)$, first of all notice that it is enough to prove the theorem for $\Gamma=\emptyset$. Secondly, to prove this simpler case, use induction on the length of the proof of $A$. Consider $A$ to be an axiom. It is easy to see that by using the same axiom in $\mathbf{GL}_h$ and using the Lemma \ref{t6-2}, we can prove the theorem. For the modus ponens case if $\mathbf{GL} \vdash A$ and $\mathbf{GL} \vdash A \rightarrow B$, by IH we have
$\mathbf{GL}_h \vdash A(w)$ and $\mathbf{GL}_h \vdash A(u) \rightarrow B(v)$ for all witnesses $w$, $u$ for $A$ and $v$ for $B$. Put $w=u$, therefore for any witness $v$ for $B$ we have $\mathbf{GL}_h \vdash B(v)$. For the necessitation case if $\mathbf{GL} \vdash A$ and $w=(n, u)$ is witness for $\Box A$ then by IH we have $\mathbf{GL}_h \vdash A(u)$, then since $n$ is bigger than all the numbers in $A(u)$, by necessitation in $\mathbf{GL}_h$ we have $\mathbf{GL}_h \vdash \Box_n A(u)$. 
\end{proof}
And as a final word in this section, we will prove that some of the logics we have introduced so far, have the strong disjunction property: 
\begin{dfn}
A modal logic $L$ in the language $\mathcal{L}_{\infty}$ has the strong disjunction property, if for any formula $\Box_n A$ and $\Box_m B$, if $L \vdash \Box_n A \vee \Box_m B$, then $L \vdash A$ or $L \vdash B$.
\end{dfn}
\begin{thm}
Logics $\mathbf{K4}_h$, $\mathbf{KD4}_h$, $\mathbf{S4}_h$ and $\mathbf{GL}_h$ have the strong disjunction property.
\end{thm}
\begin{proof}
First we prove the claim for logics $\mathbf{K4}_h$, $\mathbf{KD4}_h$ and $\mathbf{S4}_h$. The proof for all of them are the same and is based on  the usual technique of using cute-free proofs. If we denote the logic by $L$ and its sequent calculus by $G(L)$, then since $L \vdash \Box_n A \vee \Box_m B$, we know that $G(L) \vdash \; \Rightarrow \Box_n A \vee \Box_m B$. Since $G(L) \vdash \Box_n A \vee \Box_m B \Rightarrow \Box_n A, \Box_m B$, therefore by cut we have $G(L) \vdash \; \Rightarrow \Box_n A, \Box_m B$. Pick a cut-free proof of this sequent and call it $\pi$. By Theorem \ref{t2-8} we know that $\pi$ exists. Scan $\pi$ from below and find the first point that the rule is not an structural rule. Call it level $i$. Note that this $i$ exists, because if not, then by just using weakening and contraction we have to reach the axiom. It is easy to see that all valid rules in this situation are right weakening and right contraction, therefore all sequents under $i$ have the form $\Rightarrow \Box_n A, \ldots, \Box_n A, \Box_m B, \ldots, \Box_m B$. Hence we can not reach an axiom. Concludingly, $i$ exists. Moreover, since the form of all sequents under $i$ is $\Rightarrow \Box_n A, \ldots, \Box_n A, \Box_m B, \ldots, \Box_m B$, hence the rule in the $i$-th level should be a right modal rule. For that reason, the right side of the $i$-th level sequent should be a singleton which means that it should be $\Box_n A$ or $\Box_m B$. Therefore, the line above the $i$-th rule is $\Rightarrow A$ or $\Rightarrow B$ which means that we have $L \vdash A$ or $L \vdash B$.\\

For $\mathbf{GL}_h$, if $\mathbf{GL}_h \vdash \Box_n A \vee \Box_m B$, then by Theorem \ref{t6-3} part $(i)$, we have $\mathbf{GL} \vdash \Box A^f \vee \Box B^f$. Since $\mathbf{GL}$ has the strong disjunction property, we have $\mathbf{GL} \vdash A^f$ or $\mathbf{GL} \vdash B^f$. Choose $w_A$ and $w_B$ as witnesses for $A^f$ and $B^f$ such that $A^f(w_A)=A$ and $B^f(w_B)=B$. Then by Theorem \ref{t6-3} part $(ii)$, we have $\mathbf{GL}_h \vdash A^f(w_A)$ or $\mathbf{GL}_h \vdash B^f(w_B)$, which completes the proof.
\end{proof}
\section{The Logic $\mathbf{K4_h}$}
The system $\mathbf{K4}_h$ consists of the modal axioms $\mathbf{K}_h$ and $\mathbf{4}_h$. Let us investigate the intended meaning of these axioms. The axiom $\mathbf{K}_h$ actually means that the theory $T_n$ is closed under modus ponens which is what we expect of any first order arithmetical theory. The axiom $\mathbf{4}_h$ states that if $T_n \vdash A$ then $T_{n+1} \vdash \Pr_n(A)$ which is true for any strong enough theory $T_{n+1}$. Therefore, it is natural to assume that the theorems of $\mathbf{K4}_h$ should be true in all provability models. In other words, $\mathbf{K4}_h$ should be the minimum possible provability logic of hierarchies. In this section we will show this fact by proving that $\mathbf{K4}_h$ is sound and strongly complete with respect to the class of all provability models.
\begin{thm} \label{t3-1} (Soundness)
If $\Gamma \vdash_{\mathbf{K4}_h} A$ then $\mathbf{PrM} \vDash \Gamma \Rightarrow A$.
\end{thm}
\begin{proof}
To prove the soundness, we will show the following claim:\\

\textbf{Claim.} Let $(M, \{T_n\}_{n=0}^{\infty})$ be a provability model. Then if $\mathbf{K4}_h \vdash A$ then for all arithmetical substitutions $\sigma$, $I\Sigma_1 \vdash A^\sigma$.\\

The proof of the claim is by induction on the length of the proof of $A$ in $\mathbf{K4}_h$. If $A$ is a classical tautology then it is obvious that $A^{\sigma}$ is an arithmetical tautology. Therefore, $I\Sigma_1 \vdash A^\sigma$.\\ 

If $A$ is an instance of the axiom $\mathbf{H}$, then for some $B$ and some $n$, we have $A=\Box_{n} B \rightarrow \Box_{n+1} B$. Then $A^{\sigma}=\Pr_{n} (B^{\sigma}) \rightarrow \Pr_{n+1} (B^{\sigma})$. Since the hierarchy $\{T_n\}_{n=0}^{\infty}$ is provably increasing in $I\Sigma_1$, we have $I\Sigma_1 \vdash \Pr_{n} (B^{\sigma}) \rightarrow \Pr_{n+1} (B^{\sigma})$.\\

If $A$ is an instance of the axiom $\mathbf{K}_h$, then there are $B$ and $C$ and some $n$ such that $A=\Box_n (B \rightarrow C) \rightarrow (\Box_n B \rightarrow \Box_n C)$. Therefore, $A^{\sigma}=\Pr_n (B^{\sigma} \rightarrow C^{\sigma}) \rightarrow (\Pr_n (B^{\sigma}) \rightarrow \Pr_n (C^{\sigma}))$. Since the predicate $\Pr_n$ is a provability predicate then the claim holds.\\

If $A$ is an instance of the axiom $\mathbf{4}_h$, then for some $B$ and some $n$ we have $A=\Box_n B \rightarrow \Box_{n+1} \Box_n B$. Therefore, $A^{\sigma}=\Pr_n (B^{\sigma}) \rightarrow \Pr_{n+1} (\Pr_n (B^{\sigma}))$. It is provable in $I\Sigma_1$ since $\Pr_n(\cdot)$ is a $\Sigma_1$ predicate and $I\Sigma_1$ proves the $\Sigma_1$-completeness theorem for $I\Sigma_1$.  Therefore,
\[
I\Sigma_1 \vdash \Pr_n (B^{\sigma}) \rightarrow \Pr (\Pr_n (B^{\sigma}))
\]
in which $\Pr()$ is a provability predicate of $I\Sigma_1$. Since all theories $T_n$ are provably greater than $I\Sigma_1$ we have
\[
I\Sigma_1 \vdash \Pr_n (B^{\sigma}) \rightarrow \Pr_{n+1} (\Pr_n (B^{\sigma})).
\]
If $A$ is a result of the modus ponens rule, then the claim is easy to prove by using IH. And finally, if $A$ is a consequence of the necessitation rule, then there exist $B$ and $n$  such that $A=\Box_n B$. By IH, we have $I\Sigma_1 \vdash B^{\sigma}$. Since $I\Sigma_1 \subseteq T_n$, we have $T_n \vdash B^{\sigma}$. Therefore, by $\Sigma_1$-completeness and the fact that $\Pr_n(B^{\sigma})$ is a $\Sigma_1$ sentence, we have $I\Sigma_1 \vdash \Pr_n(B^{\sigma})$. \qed \\

It is easy to prove the soundness theorem by the claim. If $\Gamma \vdash_{\mathbf{K4}_h} A$, then there exists a finite $\Delta \subseteq \Gamma$ such that $\mathbf{K4}_h \vdash \bigwedge \Delta \rightarrow A $. Then by the claim, for any arithmetical substitution $\sigma$ we have 
\[
I\Sigma_1 \vdash \bigwedge \Delta^{\sigma} \rightarrow A^{\sigma}.
\]
We know that $M \vDash I\Sigma_1$. Therefore, 
\[
M \vDash \bigwedge \Delta^{\sigma} \rightarrow A^{\sigma}
\] 
which means if $M \vDash \Gamma^{\sigma}$, then $M \vDash A^{\sigma}$. 
\end{proof}
For the completeness, we will use the completeness of the translation $t$ in the previous section to reduce the completeness of $\mathbf{K4}_h$ to the completeness of $\mathbf{K4}(Q)$ which is proved in \cite{AK} and mentioned in the Preliminaries.
\begin{thm}\label{t3-2} (Strong Completeness)
If $\mathbf{PrM} \vDash \Gamma \Rightarrow A$, then $\Gamma \vdash_{\mathbf{K4}_h} A$.
\end{thm} 
\begin{proof}
We will show that $\Gamma^t \vdash_{\mathbf{K4}(Q)} A^t$. By completeness of $\mathbf{K4}(Q)$ with respect to all provability models, it is enough to prove that $\mathbf{PrM} \vDash \Gamma^t \Rightarrow A^t$. To do that, we will define a canonical witness $w_C$ for $C^t$ when $C \in \mathcal{L}_{\infty}$. Pick $n$ as a witness for some occurrence of a box in $C^t$, if that occurrence is the translation of $\Box_n$ in $C$. Since $C \in \mathcal{L}_{\infty}$, the index of the outer box is bigger than the index of the inner boxes. Therefore, $w_C$ is actually a witness. We want to show that for any provability model $(M, \{T_n\}_{n=0}^{\infty})$ and any arithmetical substitution $\sigma$,
\[ 
M \vDash \Gamma^{\sigma}(w_{\Gamma}) \Rightarrow A^{\sigma}(w_A).
\]
It is easy to see that for any arithmetical substitution $\sigma$, $(C^t)^{\sigma}(w_C)$ in the provability model $(M, \{T_n\}_{n=0}^{\infty})$ is equivalent to $C^{\sigma}$ in the provability model $(M, \{T_n+\bigwedge_{i=0}^{n}q_i^{\sigma}\}_{n=0}^{\infty})$. Since $\Gamma^{\sigma} \Rightarrow A^{\sigma}$ is true in any provability model, hence $(\Gamma^t)^{\sigma}(w_\Gamma) \Rightarrow (A^t)^{\sigma}(w)$ is true in all provability models, as well. Therefore, $\Gamma^t \vdash_{\mathbf{K4}(Q)} A^t$ and by Theorem \ref{t2-11} we have $\Gamma \vdash_{\mathbf{K4}_h} A$.
\end{proof}
\section{The Logic $\mathbf{KD4}_h$}
We know that the axioms $\mathbf{K}_h$ and $\mathbf{4}_h$ are true in all provability models. What about the axiom $\mathbf{D}_h$? It is obvious that the intended meaning of $\neg \Box_n \bot$ is the consistency of the theory $T_n$. But since we have also the proved version, i.e. $\Box_{n+1} \neg \Box_n \bot$, the axiom also implies $T_{n+1} \vdash \Cons(T_n)$. Therefore, it seems that the essence of the axiom $\mathbf{D}_h$ is the following two statements: $\Cons(T_n)$ and $T_{n+1} \vdash \Cons(T_n)$. Notice that we have the $\Sigma_1$-completeness in our theories which means that there is no need to think of $\Box_{n+2}\Box_{n+1} \neg \Box_n \bot$ or more necessitated instances of $\mathbf{D}_h$. Thus, it seems natural to guess that the logic $\mathbf{KD4}_h$ is sound and complete with respect to the class of all consistent provability models. In this section we will prove this fact.
\begin{thm}\label{t4-1}(Soundness)
If $\Gamma \vdash_{\mathbf{KD4}_h} A$, then $\mathbf{Cons} \vDash \Gamma \Rightarrow A$.
\end{thm}
\begin{proof}
To prove the theorem we need the following claim:\\

\textbf{Claim.} Let $(M, \{T_n\}_{n=0}^{\infty})$ be a consistent provability model, then if $G(\mathbf{KD4}_h) \vdash \Gamma \Rightarrow \Delta$ then for any arithmetical substitution $\sigma$, and any $n > r(\Gamma \cup \Delta)$, $M \vDash \Gamma^{\sigma} \Rightarrow \Delta^{\sigma}$ and also $M$ thinks that $T_n \vdash \Gamma^{\sigma} \Rightarrow \Delta^{\sigma}$.\\

The proof of the claim is by induction on the length of the cut-free proof of $\Gamma \Rightarrow \Delta$ in $G(\mathbf{KD4}_h)$. The axiom cases, the structural cases and the propositional cases are easy to prove. We will check just the modal rules.\\

1. If the sequent $\{\Box_n \alpha_r\}_{r \in R}, \{ \Box_{n_i} \gamma_i\}_{i \in I} \Rightarrow \Box_n A$ is proved by 
\[
\{\alpha_r\}_{r \in R}, \{ \gamma_i, \Box_{n_i} \gamma_i\}_{i \in I} \Rightarrow A
\]
and if $m>r(\{\Box_n \alpha_r\}_{r \in R}, \{ \Box_{n_i} \gamma_i\}_{i \in I} \cup \Box_n A)$, we have the following: Since $n>n_i$, we know that $n>r(\{\alpha_r\}_{r \in R}, \{ \gamma_i, \Box_{n_i} \gamma_i\}_{i \in I} \cup \{A\})$. By IH, $M$ thinks
\[
T_n \vdash \bigwedge \{\alpha_r^{\sigma}\}_{r \in R}, \{\gamma_i^{\sigma}, \Pr_{n_i} (\gamma_i^{\sigma})\}_{i \in I} \rightarrow A^{\sigma}.
\]
On the other hand, the following argument is formalizable in $I\Sigma_1$:
If 
\[
T_n \vdash \bigwedge \{\alpha_r^{\sigma}\}_{r \in R}, \{\gamma_i^{\sigma},  \Pr_{n_i} (\gamma_i^{\sigma})\}_{i \in I} \rightarrow A^{\sigma},
\]
then we have
\[
I\Sigma_1 \vdash \Pr_n(\bigwedge \alpha_r^{\sigma}) \wedge \Pr_n(\bigwedge \gamma_i^{\sigma}) \wedge \Pr_n(\bigwedge \Pr_{n_i} (\gamma_i^{\sigma})) \rightarrow \Pr_n(A^{\sigma}).
\]
Hence,
\[
I\Sigma_1 \vdash \bigwedge \Pr_n(\alpha_r^{\sigma}) \wedge \bigwedge \Pr_{n_i} (\gamma_i^{\sigma}) \rightarrow \Pr_n(A^{\sigma}) .
\]
And therefore,
\[
T_{m} \vdash \bigwedge \Pr_n(\alpha_r^{\sigma}) \wedge \bigwedge \Pr_{n_i} (\gamma_i^{\sigma}) \rightarrow \Pr_n(A^{\sigma}).
\]
The reason is that the first and the second line of the argument are easy consequences of the formalized $\Sigma_1$-completeness in $I\Sigma_1$ and provability of the fact that the provability predicate and conjunction commute and also the fact that $n_i<n$. The third line is a consequence of the fact that $I\Sigma_1 \subseteq T_m$ is formalizable in $I\Sigma_1$. Therefore, the argument is true in $M$ and hence $M$ thinks
\[
T_{m} \vdash \bigwedge \Pr_n(\alpha_r^{\sigma}) \wedge \bigwedge \Pr_{n_i} (\gamma_i^{\sigma}) \rightarrow \Pr_n(A^{\sigma})
\]
which is what we wanted. For the truth of 
\[
\bigwedge \Pr_n(\alpha_r^{\sigma}) \wedge \bigwedge \Pr_{n_i} (\gamma_i^{\sigma}) \rightarrow \Pr_n(A^{\sigma})
\]
notice that the following argument is provable in $I\Sigma_1$:
If 
\[
T_n \vdash  \bigwedge \{\alpha_r^{\sigma}\}_{r \in R}, \{\gamma_i^{\sigma},  \Pr_{n_i} (\gamma_i^{\sigma})\}_{i \in I} \rightarrow A^{\sigma}
\]
then
\[
\Pr_n(\bigwedge \alpha_r^{\sigma}) \wedge \Pr_n(\bigwedge \gamma_i^{\sigma}) \wedge \Pr_n(\bigwedge \Pr_{n_i} \gamma_i^{\sigma}) \rightarrow \Pr_n(A^{\sigma})
\]
and then
\[
\bigwedge \Pr_n(\alpha_r^{\sigma}) \wedge \bigwedge \Pr_{n_i} (\gamma_i^{\sigma}) \rightarrow \Pr_n(A^{\sigma}).
\]
Therefore, $M$ thinks that the argument is true and hence
\[
M \vDash \bigwedge \Pr_n(\alpha_r^{\sigma}) \wedge \bigwedge \Pr_{n_i} (\gamma_i^{\sigma}) \rightarrow \Pr_n(A^{\sigma}).
\]

2. For the case that the right side of the sequent is empty, it is enough to put $\bot$ for $A$ in the above proof and then we have
\[
T_{m} \vdash \bigwedge \Pr_n(\alpha_r^{\sigma}) \wedge \bigwedge \Pr_{n_i} (\gamma_i^{\sigma}) \rightarrow \Pr_n(\bot).
\]
Since $m>n$, and the provability model is consistent, we know that $M$ thinks
$
T_{m} \vdash \Cons(T_n).
$
Therefore, $M$ thinks
\[
T_m \vdash \bigwedge \Pr_n(\alpha_r^{\sigma}) \wedge \bigwedge \Pr_{n_i} (\gamma_i^{\sigma}) \rightarrow \bot
\]
which is what we wanted. For the truth part, do the same thing for the truth part above:
\[
M \vDash \bigwedge \Pr_n(\alpha_r^{\sigma}) \wedge \bigwedge \Pr_{n_i} (\gamma_i^{\sigma}) \rightarrow \Pr_n(\bot).
\]
Since the provability model is consistent, we have
$
M \vDash \Cons(T_n).
$
Hence
\[
M \vDash \bigwedge \Pr_n(\alpha_r^{\sigma}) \wedge \bigwedge \Pr_{n_i} (\gamma_i^{\sigma}) \rightarrow \bot.
\]
This completes the proof of the claim.\\

By the claim, it is easy to prove the soundness theorem. If $\Gamma \vdash_{\mathbf{KD4}_h} A$, then it is obvious by the definition that there exists a finite $\Delta \subseteq \Gamma$ such that $\Delta \vdash_{\mathbf{KD4}_h} A$. By Theorem \ref{t2-8}, there is a cut free proof of $\Delta \Rightarrow A$ in $G(\mathbf{KD4}_h)$. Hence by the claim we know that for any consistent provability model $(M, \{T_n\}_{n=0}^{\infty})$ and any arithmetical substitution $\sigma$,
$
M \vDash \Delta^{\sigma} \Rightarrow A^{\sigma}.
$
And finally since $\Delta \subseteq \Gamma$, we have
$
M \vDash \Gamma^{\sigma} \Rightarrow A^{\sigma}
$
which completes the proof.
\end{proof}
The completeness is an easy consequence of the completeness of $\mathbf{K4}_h$. 
\begin{thm}\label{t4-2}(Strong Completeness)
If $\mathbf{Cons} \vDash \Gamma \Rightarrow A$, then $\Gamma \vdash_{\mathbf{KD4}_h} A$.
\end{thm}
\begin{proof}
Let $\Delta$ be the set of all instances of the schemes $\neg \Box_n \bot$ and $\Box_{n+1} \neg \Box_n \bot$. Then it is easy to see that $\mathbf{PrM} \vDash \Gamma + \Delta \Rightarrow A$. The reason is that if a provability model $(M, \{T_n\}_{n=0}^{\infty})$ is a model of $\Delta$, then it should be a consistent provability model. Therefore, by strong completeness for $\mathbf{K4}_h$, we have $\Gamma+ \Delta \vdash_{\mathbf{K4}_h} A $. Since $\mathbf{KD4}_h$ proves all formulas in $\Delta$, we have $\Gamma \vdash_{\mathbf{KD4}_h} A$. 
\end{proof}
\section{The Logic $\mathbf{S4}_h$}
What is the intended meaning of the axiom $\mathbf{T}_h$? It is easy to see that $\Box_n A \rightarrow A$ means that $T_n$ is sound. But we also have $\Box_{n+1} (\Box_n A \rightarrow A)$ which means that the soundness of $T_n$ should be provable in $T_{n+1}$. Similar to the case of $\mathbf{KD4}_h$, there is no need to worry about more applications of necessitation. Therefore, it seems that the natural canonical models for $\mathbf{S4}_h$ are the reflexive provability models. In this section we will show that $\mathbf{S4}_h$ is sound and strongly complete with respect to the class of all reflexive provability models.
\begin{thm}\label{t5-1}(Soundness)
If $\Gamma \vdash_{\mathbf{S4}_h} A$, then $\mathbf{Ref} \vDash \Gamma \Rightarrow A$.
\end{thm}
\begin{proof}
To prove the theorem, we will prove the following claim.\\

\textbf{Claim.} Let $(M, \{T_n\}_{n=0}^{\infty})$ be a reflexive provability model, then if $G(\mathbf{S4}_h) \vdash \Gamma \Rightarrow \Delta$ then for any arithmetical substitution $\sigma$, and any $n > r(\Gamma \cup \Delta)$, $M$ thinks that $T_n \vdash \Gamma^{\sigma} \Rightarrow \Delta^{\sigma}$.\\

The proof of the claim is by induction on the length of the cut-free proof of $\Gamma \Rightarrow \Delta$ in $G(\mathbf{S4}_h)$. The axiom cases, the structural cases and the propositional cases are easy to prove. We will check just the modal rules.\\

1. If the sequent $\Gamma, \Box_n A \Rightarrow \Delta$ is proved by $\Gamma, A \Rightarrow \Delta$ and $m>r(\Gamma, \Box_n A \cup \Delta)$, then $m>r(\Gamma, A \cup \Delta)$. Therefore by IH, we know that $M$ thinks that
\[
T_m \vdash \Gamma^{\sigma}, A^{\sigma} \Rightarrow \Delta^{\sigma} .
\]
On the other hand, the following fact is formalizable in $I\Sigma_1$:
If 
\[
T_m \vdash \Gamma^{\sigma}, A^{\sigma} \Rightarrow \Delta^{\sigma} 
\] 
and
\[
T_{m} \vdash \Pr_n(A^{\sigma}) \rightarrow A^{\sigma}
\]
then 
\[
T_{m} \vdash\Gamma^{\sigma}, \Pr_n(A^{\sigma}) \Rightarrow \Delta^{\sigma}. 
\]
Therefore, the argument is true in $M \vDash I\Sigma_1$. Since $m>n$ and the provability model is reflexive, both of the assumptions are true in $M$. Hence, the conclusion is true in $M$.\\

2. If the sequent $\{\Box_n \alpha_r\}_{r \in R}, \{ \Box_{n_i} \gamma_i\}_{i \in I} \Rightarrow \Box_n A$ is proved by 
\[
\{\alpha_r\}_{r \in R}, \{ \Box_{n_i} \gamma_i\}_{i \in I} \Rightarrow A
\]
and if $m>r(\{\Box_n \alpha_r\}_{r \in R}, \{ \Box_{n_i} \gamma_i\}_{i \in I} \cup \Box_n A)$ we have the following: Since $n>n_i$, we know that $n>r(\{\alpha_r\}_{r \in R}, \{ \Box_{n_i} \gamma_i\}_{i \in I} \cup \{A\})$. By IH, $M$ thinks
\[
T_n \vdash  \bigwedge \{\alpha_r^{\sigma}\}_{r \in R}, \{ \Pr_{n_i} (\gamma_i^{\sigma})\}_{i \in I} \rightarrow A^{\sigma}.
\]
On the other hand, the following argument is formalizable in $I\Sigma_1$:
If 
\[
T_n \vdash \bigwedge \{\alpha_r^{\sigma}\}_{r \in R}, \{ \Pr_{n_i} (\gamma_i^{\sigma})\}_{i \in I} \rightarrow A^{\sigma}
\]
then
\[
T_{m} \vdash \bigwedge \Pr_n(\alpha_r^{\sigma} \wedge \bigwedge \Pr_{n_i} (\gamma_i^{\sigma})) \rightarrow \Pr_n(A^{\sigma})
\]
and hence
\[
T_{m} \vdash \bigwedge \Pr_n(\alpha_r^{\sigma}) \wedge \bigwedge \Pr_{n_i} (\gamma_i^{\sigma}) \rightarrow \Pr_n(A^{\sigma}).
\]
The reason is that the first and the second line of the argument are easy consequences of the formalized $\Sigma_1$-completeness in $I\Sigma_1$ and the fact that $\Pr_n$ commutes with conjunction provably in $I\Sigma_1$. Therefore, the argument is true in $M$ and hence $M$ thinks
\[
T_{m} \vdash \bigwedge \Pr_n(\alpha_r^{\sigma}) \wedge \bigwedge \Pr_{n_i} (\gamma_i^{\sigma}) \rightarrow \Pr_n(A^{\sigma})
\]
which is what we wanted.\\

By the claim, it is easy to prove the soundness theorem. If $\Gamma \vdash_{\mathbf{S4}_h} A$, then it is obvious by the definition that there exists a finite $\Delta \subseteq \Gamma$ such that $\Delta \vdash_{\mathbf{S4}_h} A$. By Theorem \ref{t2-8}, there is a cut free proof of $\Delta \Rightarrow A$ in $G(\mathbf{S4}_h)$. Hence if we choose a natural number $n$ bigger than $r(\Delta \cup \{A\})$ then by the claim we know that for any reflexive provability model $(M, \{T_n\}_{n=0}^{\infty})$ and any arithmetical substitution $\sigma$, $M$ thinks that
$
T_n \vdash \Delta^{\sigma} \Rightarrow A^{\sigma}
$
which means
$
M \vDash \Pr_n(\bigwedge \Delta^{\sigma} \rightarrow A^{\sigma})
$.
But the provability model is a reflexive model, therefore,
$M \vDash \Pr_n(\phi) \rightarrow \phi$, hence
$
M \vDash \Delta^{\sigma} \Rightarrow A^{\sigma}
$.
Finally since $\Delta \subseteq \Gamma$,
$
M \vDash \Gamma^{\sigma} \Rightarrow A^{\sigma}
$
which completes the proof.
\end{proof}
To prove the completeness, similar to the case $\mathbf{K4}_h$, we will reduce the completeness of $\mathbf{S4}_h$ to the completeness of $\mathbf{S4}(Q)$ which is proved in \cite{AK} and also mentioned in the Preliminaries.
\begin{thm}\label{t5-2}(Strong Completeness)
If $\mathbf{Ref} \vDash \Gamma \Rightarrow A$, then $\Gamma \vdash_{\mathbf{S4}_h} A$.
\end{thm}
\begin{proof}
Let $ \{T_n\}_{n=0}^{\infty}$ be a uniformly reflexive hierarchy of sound theories. And assume that $*$ is a uniform substitution of the Theorem \ref{t1-7}. Remind from the proof of the Theorem \ref{t3-2} that we have a canonical witness $w_C$ for $C^t$ when $C \in \mathcal{L}_{\infty}$. It was defined as follows: Pick $n$ as a witness for some occurrence of a box in $C^t$, if that occurrence is the translation of $\Box_n$ in $C$. First of all, we want to show that there exists a finite $\Delta \subseteq \Gamma$ and $m$ greater than all the numbers in $w_{\Delta}$ and $w_A$ such that 
\[
T_m + (\Delta^t)^*(w_{\Delta}) \vdash \bigwedge_{i=0}^{m}q_i^* \rightarrow (A^t)^{*}(w_A).
\]
Pick a model $M \vDash \bigcup_{n=0}^{\infty} T_n +\{q_i^*\}_{i=0}^{\infty}$. Then $(M, \{T_n+ \bigwedge_{i=0}^{n} q_i^*\}_{n=0}^{\infty})$ is a reflexive provability model. Therefore, $M \vDash \Gamma^* \Rightarrow A^*$. It is easy to check that for any formula $B$, $B^*$ with respect to the provability model $(M, \{T_n+ \bigwedge_{i=0}^{n} q_i^*\}_{n=0}^{\infty})$ is equivalent to $(B^t)^*(w_B)$ with respect to the provability model
$(M, \{T_n\}_{n=0}^{\infty})$. Therefore, $(M, \{T_n\}_{n=0}^{\infty}) \vDash (\Gamma^t)^*(w_{\Gamma}) \Rightarrow (A^t)^{*}(w_A)$. Since it is true for all $M \vDash \bigcup_{n=0}^{\infty} T_n +\{q_i^*\}_{i=0}^{\infty}$, we have
\[
\bigcup_{n=0}^{\infty} T_n +\{q_i^*\}_{i=0}^{\infty} + (\Gamma^t)^*(w_{\Gamma}) \vdash (A^t)^{*}(w_A).
\]
Therefore there exists some $m$ and some finite $\Delta \subseteq \Gamma$ such that
\[
T_m + (\Delta^t)^*(w_{\Delta}) \vdash \bigwedge_{i=0}^{m}q_i^* \rightarrow (A^t)^{*}(w_A).
\] 
Notice that w.l.o.g we can choose $m$ big enough to reach the condition that $m$ should be greater that all the numbers in $w_{\Delta}$ and $w_A$. Hence,
\[
I\Sigma_1 \vdash \Pr_m(\bigwedge_{i=0}^{m}q_i^* \rightarrow (\bigwedge (\Delta^t)^*(w_{\Delta}) \rightarrow (A^t)^{*}(w_A))).
\]
But 
\[
\Pr_m(\bigwedge_{i=0}^{m}q_i^* \rightarrow (\bigwedge (\Delta^t)^*(w_{\Delta}) \rightarrow (A^t)^{*}(w_A)))
\]
is $((\Box_m(\bigwedge \Delta \rightarrow A))^t)^*(w)$, where $w=w_{(\Box_m(\bigwedge \Delta \rightarrow A))^t}$.
Therefore,
$
((\Box_m(\bigwedge \Delta \rightarrow A))^t)^*(w)
$
is true in all the models of the form $(M, \{T_n\}_{n=0}^{\infty})$ in which $M \vDash \bigcup_{n=0}^{\infty} T_n$. By strong uniform completeness for $\mathbf{S4}(Q)$, Theorem \ref{t1-7}, we have $\mathbf{S4}(Q) \vdash (\Box_m(\bigwedge \Delta \rightarrow A))^t$. Then by Theorem \ref{t2-11}, we have $\mathbf{S4}_h \vdash \Box_m( \bigwedge \Delta \rightarrow A)$. Hence $\mathbf{S4}_h \vdash \bigwedge \Delta \rightarrow A$ which means that $\Gamma \vdash_{\mathbf{S4}_h} A$.
\end{proof}
\section{The Logic $\mathbf{GL}_h$}
In this section we will show that the usual provability logic approach to investigate the provability-based behavior of theories (instead of hierarchies) is a special case of this new framework. It is enough to limit ourselves to the constant hierarchies.
\begin{thm}\label{t6-1}(Soundness)
If $\Gamma \vdash_{\mathbf{GL}_h} A$, then $\mathbf{Cst} \vDash \Gamma \Rightarrow A$.
\end{thm}
\begin{proof}
By definition it is enough to show that if $\mathbf{GL}_h \vdash A$ then $\mathbf{Cst} \vDash A$. To do that we prove the following claim:\\

\textbf{Claim.} Let $(M, \{T_n\}_{n=0}^{\infty})$ be a constant provability model and $\mathbf{GL}_h \vdash A$. Then for any arithmetical substitution $\sigma$, $M \vDash A^{\sigma}$ and also $M$ thinks that $T_0 \vdash A^{\sigma}$.\\

The proof of the claim is by induction on the length of the proof of $A$. First of all the axioms. Since the substituted version of all axioms of $\mathbf{GL}_h$ except $\mathbf{L}_h$ are provable in $I\Sigma_1$ (see the proof of Theorem \ref{t3-1}), it is easy to see that the claim holds. If $A$ is an instance of the L\"{o}b axiom, then $A=\Box_{n+1}(\Box_n B \rightarrow B)\rightarrow \Box_n B$. By the formalized L\"{o}b's theorem for $T_n$, we have
\[
I\Sigma_1 \vdash \Pr_{n}(\Pr_n (B^{\sigma}) \rightarrow B^{\sigma})\rightarrow \Pr_n (B^{\sigma}) .
\] 
Therefore it is true in $M$, and since $M$ thinks that $T_n=T_{n+1}$, we have
\[
M \vDash \Pr_{n+1}(\Pr_n (B^{\sigma}) \rightarrow B^{\sigma})\rightarrow \Pr_n (B^{\sigma}).
\]
On the other hand, we know that the following is formalizable in $I\Sigma_1$:
If
\[
T_0 \vdash \Pr_{n}(\Pr_n (B^{\sigma}) \rightarrow B^{\sigma})\rightarrow \Pr_n (B^{\sigma}) \;\;\; (*)
\]
and
\[
T_0 \vdash \Pr_{n+1}(\Pr_n (B^{\sigma}) \rightarrow B^{\sigma}) \leftrightarrow \Pr_n(\Pr_n (B^{\sigma}) \rightarrow B^{\sigma})
\]
then 
\[
T_0 \vdash \Pr_{n+1}(\Pr_n (B^{\sigma}) \rightarrow B^{\sigma}) \rightarrow \Pr_n (B^{\sigma}) .
\]
Moreover, the line $(*)$ is a $\Sigma_1$ true statement, hence it is provable in $I\Sigma_1$. Therefore, it is formalizable in $I\Sigma_1$. Hence the following is formalizable in $I\Sigma_1$:
If
\[
T_0 \vdash \Pr_{n+1}(\Pr_n (B^{\sigma}) \rightarrow B^{\sigma}) \leftrightarrow \Pr_n(\Pr_n (B^{\sigma}) \rightarrow B^{\sigma})
\]
then 
\[
T_0 \vdash \Pr_{n+1}(\Pr_n (B^{\sigma}) \rightarrow B^{\sigma}) \rightarrow \Pr_n (B^{\sigma}) .
\]
Therefore, this fact is true in $M \vDash I\Sigma_1$, and since the provablity model is constant, $M$ thinks that 
\[
T_0 \vdash \Pr_{n+1}(\Pr_n (B^{\sigma}) \rightarrow B^{\sigma}) \leftrightarrow \Pr_n(\Pr_n (B^{\sigma}) \rightarrow B^{\sigma}).
\] 
Hence $M$ thinks that
\[
T_0 \vdash \Pr_{n+1}(\Pr_n (B^{\sigma}) \rightarrow B^{\sigma}) \rightarrow \Pr_n (B^{\sigma}). 
\] 
The case of modus ponens is easy to see. For the necessitation, if $\Box_n A$ is proved by $A$, then by IH we know that $M$ thinks that
$T_0 \vdash A^{\sigma} $. Therefore by the formalized $\Sigma_1$-completeness in $I\Sigma_1$, we know that $M$ thinks
$
I\Sigma_1 \vdash \Pr_0(A^{\sigma})
$
and since $n \geq 0$, then $M$ thinks $I\Sigma_1 \vdash \Pr_n(A^{\sigma})$. And since $I\Sigma_1 \subseteq T_0$ provably in $I\Sigma_1$, we know that 
$
T_0 \vdash \Pr_n(A^{\sigma}).
$
To prove the truth of $\Pr_n(A^{\sigma})$ in $M$, by IH,  $M$ thinks that
$
T_0 \vdash A^{\sigma}.
$ 
Therefore,
$
M \vDash \Pr_0(A^{\sigma})
$
and since $n \geq 0$, we have 
$M \vDash \Pr_n(A^{\sigma})$.\\

It is clear that the theorem is a consequence of the claim.
\end{proof}
\begin{thm}\label{t6-4}(Strong Completeness)
Let $T$ be a $\Sigma_1$-sound theory. Then if for all models $M \vDash T$ we have $(M, \{T\}_{n=0}^{\infty}) \vDash \Gamma \Rightarrow A$, then $\Gamma \vdash_{\mathbf{GL}_h} A$. Specially, if $\mathbf{Cst} \vDash \Gamma \Rightarrow A$ then $\Gamma \vdash_{\mathbf{GL}_h} A$.
\end{thm}
\begin{proof}
Let us remind the forgetful translation from the third section. The forgetful translation $f: \mathcal{L}_{\infty} \to \mathcal{L}_{\Box}$ translates atomic formulas and propositional connectives to themselves and sends $\Box_n$ to $\Box$. Moreover, we know that there exists some witness $w_B$ for $B^f$ such that $B^f(w_B)=B$. Since all theories in the hierarchy of the provability models $(M, \{T\}_{n=0}^{\infty}) \vDash \Gamma \Rightarrow A$ are equal, we can conclude that for all arithmetical substitutions $\sigma$ and all formula $B$, $M \vDash B^{\sigma} \leftrightarrow (B^f)^{\sigma}(w_B)$. Therefore, $M \vDash (\Gamma^f)^{\sigma}(w_{\Gamma}) \Rightarrow (A^f)^{\sigma}(w_A)$. Since $\Gamma^f \cup \{A^f\} \subseteq \mathcal{L}_{\Box}$, by Theorem \ref{t1-7}, we have $\Gamma^f \vdash_{\mathbf{GL}} A^f$. By Theorem \ref{t6-3}, $\Gamma^f(w_{\Gamma}) \vdash_{\mathbf{GL}_h} A^f(w_A)$. Hence, $\Gamma \vdash_{\mathbf{GL}_h} A$.
\end{proof}
\section{The Extensions of $\mathbf{KD45}_h$}
As expected, the axiom $\mathbf{5}_h$ in the presence of the axioms $\mathbf{T}_h$ and $\mathbf{4}_h$ are too strong to have a provability interpretation. It informally means that the theory $T_n$ should be decidable in $T_{n+1}$ which implies the decidability of $T_n$. (See \cite{AK}). In this section we will show a more strong version which states that there is no provability interpretation for the extensions of $\mathbf{KD45}_h$.
\begin{thm}\label{t7-1}
There is no provability model $(M, \{T_n\}_{n=0}^{\infty})$ such that 
\[
(M, \{T_n\}_{n=0}^{\infty}) \vDash \mathbf{KD45}_h.
\]
Hence, there are no provability models for any extension of the logic $\mathbf{KD45}_h$. Specially, $\mathbf{S5}_h$ does not have any provability interpretation.
\end{thm}
\begin{proof}
Assume that $(M, \{T_n\}_{n=0}^{\infty}) \vDash \mathbf{KD45}_h$. Then we know that for any arithmetical substitution $\sigma$ we have
\[
M \vDash \neg \Pr_n (p^{\sigma}) \rightarrow \Pr_{n+1} (\neg \Pr_n (p^\sigma)).
\]
Pick an arithmetical substitution which send $p$ to $\Pr_{n+1}(\bot)$. Therefore, we have
\[
M \vDash \neg \Pr_n (\Pr_{n+1}(\bot)) \rightarrow \Pr_{n+1} (\neg \Pr_n (\Pr_{n+1}(\bot))). \;\;\; (*)
\]
On other hand by the formalized $\Sigma_1$-completeness, we have 
\[
I\Sigma_1 \vdash \neg \Pr_n(\Pr_{n+1}(\bot)) \rightarrow \neg \Pr_{n+1}(\bot).
\]
Hence,
\[
T_{n+1} \vdash \neg \Pr_n(\Pr_{n+1}(\bot)) \rightarrow \neg \Pr_{n+1}(\bot).
\]
Moreover, by $\Sigma_1$-completeness, we have 
\[
I\Sigma_1 \vdash \Pr_{n+1}(\neg \Pr_n(\Pr_{n+1}(\bot)) \rightarrow \neg \Pr_{n+1}(\bot)).
\]
Therefore,
\[
I\Sigma_1 \vdash \Pr_{n+1}(\neg \Pr_n(\Pr_{n+1}(\bot))) \rightarrow \Pr_{n+1}(\neg \Pr_{n+1}(\bot)).
\]
And since $M \vDash I\Sigma_1$, we have
\[
M \vDash \Pr_{n+1}(\neg \Pr_n(\Pr_{n+1}(\bot))) \rightarrow \Pr_{n+1}(\neg \Pr_{n+1}(\bot)).
\]
Therefore by $(*)$ we have
\[
M \vDash \neg \Pr_n(\Pr_{n+1}(\bot)) \rightarrow \Pr_{n+1}(\neg \Pr_{n+1}(\bot)).
\]
Based on G\"{o}del's second incompleteness theorem formalized in $I\Sigma_1$, we can conclude
\[
I\Sigma_1 \vdash \neg \Pr_{n+1}(\bot) \rightarrow \neg \Pr_{n+1}(\neg \Pr_{n+1}(\bot)).
\]
Therefore, 
\[
M \vDash \neg \Pr_n(\Pr_{n+1}(\bot)) \rightarrow \Pr_{n+1}(\bot).
\]
But we know that the provability model $(M, \{T_n\}_{n=0}^{\infty}) $ is a model for $\mathbf{D}_h$, therefore $M \vDash \neg \Pr_{n+1} (\bot)$. Hence,
$
M \vDash \Pr_n(\Pr_{n+1}(\bot))
$.
Since the theories are provably increasing, we have
$
M \vDash \Pr_{n+2}(\Pr_{n+1}(\bot))
$.
Again, since the provability model is a model for the logic $\mathbf{KD4}_h$, therefore, it is a model of the formula $\Box_{n+2} \neg \Box_{n+1} \bot$. Hence,
$
M \vDash \Pr_{n+2}(\neg \Pr_{n+1}(\bot))
$.
Therefore,
$
M \vDash \Pr_{n+2}(\bot)
$,
which contradicts with an instance of the axiom $\mathbf{D}_h$.
\end{proof}
\vspace{4pt}
\textbf{Acknowledgment.} We wish to thank Pavel Pudl\'{a}k for all the helpful discussions, specifically pointing out the importance of using the poly-modal language. Also, we are grateful to Raheleh Jalali and Masoud Memarzadeh for their careful reading of the earlier draft and their useful comments.


\begin{thebibliography}{99} 
\addcontentsline{toc}{section}{References} 
\bibitem{AK}
A. Akbar tabatabai, Provability Interpretation of Propositional and Modal Logics, Preprint, 2016.
\bibitem{G}
K. G\"{o}del, Eine Interpretation des Intuitionistichen Aussagenkalk\"{u}ls, Ergebnisse Math Colloq. Vol. 4 (1933), pp. 39-40.
\bibitem{Po}
F. Poggiolesi, Gentzen Calculi for Modal Propositional Logic, Springer, 2010.
\bibitem{So}
R. Solovay, Provability interpretations of modal logic, Israel Journal of Mathematics,
vol. 25 (1976), pp. 287-304.
\end{thebibliography}
\end{document}